\documentclass{amsart}
\usepackage{amsmath,amsthm,amssymb}
\usepackage{hyperref,graphicx,color}

\usepackage{a4wide}
\usepackage{soul}

\title{On $q$-integrals over order polytopes}

\author{Jang Soo Kim}
\address[Jang Soo Kim]{
Department of Mathematics, Sungkyunkwan University, Suwon, South Korea}
\email{jangsookim@skku.edu}

\author{Dennis Stanton}
\address[Dennis Stanton]{School of Mathematics, University of Minnesota, Minneapolis, MN 55455}
\email{stanton@math.umn.edu}
\date{\today}

\thanks{The first author was partially supported by Basic Science
  Research Program through the National Research Foundation of Korea
  (NRF) funded by the Ministry of Education (NRF-2013R1A1A2061006).
  The second author was supported by NSF grant DMS-1148634}



\newtheorem{thm}{Theorem}[section]
\newtheorem{lem}[thm]{Lemma}
\newtheorem{prop}[thm]{Proposition}
\newtheorem{cor}[thm]{Corollary}
\theoremstyle{definition}
\newtheorem{exam}[thm]{Example}
\newtheorem{defn}[thm]{Definition}

\newtheorem{remark}[thm]{Remark}

\setstcolor{red}

\newcommand\Qbinom[3]{\genfrac{[}{]}{0pt}{}{#1}{#2}_{#3}}
\newcommand\qbinom[2]{\Qbinom{#1}{#2}{q}}
\newcommand\LL{\mathcal{L}}
\newcommand\GT{\mathrm{GT}}

\newcommand\order{\mathcal{O}}
\newcommand\Par{\mathrm{Par}}
\newcommand\DeltaBar{\overline{\Delta}}

\newcommand\inv{\operatorname{inv}}
\newcommand\maj{\operatorname{maj}}
\newcommand\Des{\operatorname{Des}}
\newcommand\des{\operatorname{des}}
\newcommand\wt{\operatorname{wt}}

\renewcommand\Re{\operatorname{Re}}

\newcommand\rdiag{\operatorname{rdiag}}
\newcommand\tr{\operatorname{tr}}
\newcommand\RPP{\operatorname{RPP}}

\newcommand\schur{\operatorname{Schur}}

\renewcommand\vec[1]{{#1}}

\def\JV.{Josuat-Verg\`es}

\begin{document}

\begin{abstract}
A combinatorial study of multiple $q$-integrals is developed.
This includes a $q$-volume of a convex polytope, which 
depends upon the order of $q$-integration. 
A multiple $q$-integral over an order polytope of a poset is
interpreted as a generating function of linear extensions of the
poset. Specific modifications of posets are shown to give 
predictable changes in $q$-integrals over 
their respective order polytopes. 
This method is used to combinatorially evaluate some generalized 
$q$-beta integrals. One such application is a combinatorial 
interpretation of a $q$-Selberg integral. 
New generating functions for generalized Gelfand-Tsetlin patterns
and reverse plane partitions are established. A $q$-analogue to a 
well known result in Ehrhart theory is generalized using 
$q$-volumes and $q$-Ehrhart polynomials.
\end{abstract}

\maketitle
\tableofcontents

\newcommand\dqx{d_q x_1\cdots d_q x_n}
\newcommand\DQ[2]{d_q {#1}_1\cdots d_q {#1}_{#2}}
\newcommand\id{\mathrm{id}}

\section{Introduction}

The main object in this paper is the $q$-integral
\[
\int_0^1 f(x) d_qx = (1-q)\sum_{i=0}^n f(q^i) q^i, 
\]
which was introduced by Thomae \cite{Thomae} and Jackson \cite{Jackson}. 
The $q$-integral is a $q$-analogue of the Riemann integral. 
Fermat used it to evaluate $\int_0^1 x^n dx$. See \cite[\S10.1]{AAR_SP} 
for more details of the history of $q$-integrals.  
Many important integrals have
$q$-analogues in terms of $q$-integrals, such as $q$-beta integrals
and $q$-Selberg integrals. In this paper we develop combinatorial
methods to study $q$-integrals.

The original motivation of this paper was to generalize Stanley's
combinatorial interpretation of the Selberg integral \cite{Selberg1944}
\begin{multline*} 
\int_0^1\cdots\int_0^1
\prod_{i=1}^n x_i^{\alpha-1} (1-x_i)^{\beta-1}
\prod_{1\le i<j\le n} |x_i-x_j|^{2\gamma} dx_1\cdots dx_n\\ 
=\prod_{j=1}^n \frac{\Gamma(\alpha+(j-1)\gamma)
\Gamma(\beta+(j-1)\gamma)\Gamma(1+j\gamma)}
{\Gamma(\alpha+\beta+(n+j-2)\gamma)\Gamma(1+\gamma)},
\end{multline*}
where $n$ is a positive integer and $\alpha,\beta,\gamma$ are complex
numbers such that $\mathrm{Re} (\alpha)>0$, $\mathrm{Re} (\beta)>0$,
and
$\mathrm{Re} (\gamma)>
-\min\{1/n,\mathrm{Re}(\alpha)/(n-1),\mathrm{Re}(\beta)/(n-1)\}$.
Stanley \cite[Exercise 1.10 (b)]{EC1} found a combinatorial
interpretation of the above integral when $\alpha-1,\beta-1$ and
$2\gamma$ are nonnegative integers in terms of permutations. His idea
is to interprete the integral as the probability that a random
permutation satisfies certain properties. This idea uses the fact that
a real number $x\in (0,1)$ can be understood as the probability that a
random number selected from $(0,1)$ lies on an interval of length $x$
is equal to $x$. 

In order to find a combinatorial interpretation of a $q$-analogue of
this integral, see \eqref{eq:Askey}, we take a different approach. We
interpret $q$-integrals as generating functions in $q$. This is not
surprising, as the $q$-integral itself is a sum.  Here is a brief
summary of our approach to this problem. We will define a $q$-volume
of a polytope by a certain multiple $q$-integral. The polytopes of
interest are order polytopes of posets. We shall see that simple
operations on posets correspond to insertions of polynomials in the
integrands of the multiple $q$-integral. Using these simple
operations, we can define a poset whose order polytope has a
$q$-volume given by a $q$-Selberg integral. We show that the
$q$-volume of an order polytope is a generating function for linear
extensions of the poset. This gives a combinatorial interpretation of
a $q$-Selberg integral, see Corollary~\ref{cor:majSelberg}.

The purpose of this paper is not limited to answering the motivational
question on the Selberg integral. We have examples and
applications of the combinatorial methods developed in this paper.

The key property of $q$-integrals is the failure of Fubini's theorem,
but in a controlled way, see Corollary~\ref{cor:maj}.  We show in
Theorem~\ref{thm:pw} that $q$-volume of the order polytope of a poset $P$ is equal to a
generating function for $(P,\omega)$-partitions, where the labeling
$\omega$ of the poset $P$ corresponds to the order of
integration. Equivalently, using a well known fact in
$(P,\omega)$-partition theory due to Stanley \cite{Stanley72}, this is
equal to a generating function for the linear extensions of $P$.

The remainder of this paper is organized as follows. 

In Section~\ref{sec:defn} we give definitions that are used throughout
the paper.

In Section~\ref{sec:prop-q-integr} we prove basic properties of the
$q$-integrals. We investigate how Fubini's theorem fails and when it
holds. We give an expansion formula for the $q$-integral over a
polytope determined by certain inequalities.

In Section~\ref{sec:q-integrals-over} we study $q$-integrals over
order polytopes of posets. We show that the $q$-volume of the order
polytope of a poset is the \emph{maj}-generating function for the
linear extensions of the poset, up to a constant factor.

In Section~\ref{sec:operations-posets} we consider simple operations
on $P$ such as adding a new chain. We show how the $q$-integral
changes over the order polytope of $P$ when these operations are performed.

In Section~\ref{sec:examples}, using the results in the previous
sections, we consider several $q$-integrals: the $q$-beta integral, a
$q$-analogue of Dirichlet's integral, a generalized $q$-beta integral
due to Andrews and Askey \cite{AndrewsAskey1981}.  

In Section~\ref{sec:q-selberg-integrals} we give a combinatorial
interpretation of Askey's $q$-Selberg integral in terms of a
generating function of the linear extensions of a poset. 

In Section~\ref{sec:reverse-plane-part} we study reverse plane
partitions using Selberg-type $q$-integrals which involve Schur
functions. We show that these $q$-integrals are essentially generating
functions of reverse plane partitions with certain weights. By using
known evaluation formulas for these $q$-integrals we obtain a formula
for the generating function for the reverse plane partitions with
fixed shape (both shifted and normal) and fixed diagonal entries. This
can be restated as a generating function for generalized
Gelfand-Tsetlin patterns.  Taking the sum of these generating
functions yields the well known trace-generating function formulas for
reverse plane partitions of fixed (shifted or normal) shape. We also
show that Askey's $q$-Selberg integral is equivalent to a new generating
function for reverse plane partitions of a square shape. This allows
us to obtain a new product formula for the generating function for
reverse plane partitions of a square shape with a certain weight.

In Section~\ref{sec:q-ehrh-polyn} we study $q$-Ehrhart polynomials and
$q$-Ehrhart series of order polytopes using $q$-integrals. We show
that the $q$-Ehrhart function of an order polytope is a polynomial in
a particular sense whose leading coefficient is the $q$-volume of the
order polytope of the dual poset.

\section{Definitions}
\label{sec:defn}

In this section we give the necessary definitions with examples for 
$q$-integration, multiple $q$-integration, and order polytopes of posets. 

Throughout this paper we assume $0<q<1$. 
We will use the following notation for $q$-series:
\[
[n]_q = \frac{1-q^n}{1-q}, \qquad [n]_q! = [1]_q[2]_q\cdots[n]_q,
\qquad \qbinom{n}{k}= \frac{[n]_q!}{[k]_q![n-k]_q!},
\]
\[
(a;q)_n = (1-a)(1-aq)\cdots(1-aq^{n-1}),
\qquad (a_1,a_2,\dots,a_k;q)_n=(a_1;q)_n \cdots (a_k;q)_n.
\]

We also use the notation $[n]:=\{1,2,\dots,n\}$. We denote by $S_n$
the set of permutations on $[n]$.
\begin{defn} Let $\pi=\pi_1\pi_2\cdots\pi_n\in S_n$. An integer $i\in[n-1]$ is
called a \emph{descent} of $\pi$ if $\pi_i>\pi_{i+1}$. Let $\Des(\pi)$
be the set of descents of $\pi$.  We define $\des(\pi)$ and
$\maj(\pi)$ to be the number of descents of $\pi$ and the sum of
descents of $\pi$, respectively. We denote by $\inv(\pi)$ the number of
pairs $(i,j)$ such that $1\le i<j\le n$ and $\pi_i>\pi_j$. 
\end{defn}

First, recall \cite[\S10.1]{AAR_SP} the $q$-integral of 
a function $f$ over $(a,b)$.
\begin{defn} For $0<q<1$, the \emph{$q$-integral from $a$ to $b$} is defined by
\begin{equation}
\label{q-intdefn}
\int_a^b f(x) d_qx = (1-q)\sum_{i=0}^\infty 
\left( f(bq^i) bq^i - f(aq^i) aq^i \right).
\end{equation}
\end{defn}

In the limit as $q\to 1$, the $q$-integral becomes the usual integral.
It is easy to see that
\[
\int_a^b x^n d_qx = \frac{b^{n+1}-a^{n+1}}{[n+1]_q}.
\]

We extend the definition of a $q$-integral to a multiple $q$-integral
over a convex polytope. Here the ordering $\pi$ of the variables is important because 
the iterated $q$-integral is not, in general, independent of the ordering. 

\begin{defn}
\label{defn:polytope}
  Let $\pi=\pi_1\cdots\pi_n\in S_n$.  For a function $f(\vec x)$ of
  $n$-variables $\vec x=(x_1,\dots,x_n)$ and a convex polytope
  $D\in\mathbb{R}^n$, \emph{the $q$-integral of $f(\vec x)$ over $D$
    with respect to order $\pi$ of integration} is defined by
\[
\int_D f(x_1,\dots,x_n) d_q x_{\pi_1} \cdots d_q x_{\pi_n} = 
\int_{\min(D_n)}^{\max(D_n)} \cdots \int_{\min(D_1)}^{\max(D_1)}
f(x_1,\dots,x_n) d_q x_{\pi_1} \cdots d_q x_{\pi_n},
\]
where $D_i$ is the set of real numbers depending on the values of
$x_{\pi_{i+1}},\dots,x_{\pi_n}$ given by
\[
D_i = D_i(x_{\pi_{i+1}},\dots,x_{\pi_n})=\{y_{\pi_i}: (y_1,\dots,y_n)\in D,\quad 
y_{\pi_j}=x_{\pi_j} \text{ for } j>i\}.
\]
\end{defn}

If a convex polytope $D$ is determined by a family $Q$ of inequalities of
the sorted variables $\vec x=(x_1,\dots,x_n)$, then we will also write
\[
\int_{Q} f(\vec x) d_q\vec x=\int_{D} f(\vec x) d_q\vec x,
\qquad {\text{where }} d_q\vec x= d_qx_{1}\cdots d_qx_{n}.
\]

\begin{exam}
We have 
\[
\int_{\substack{a\le x\le b\\ c\le y\le d}}f(x,y) d_qx d_qy =
\int_{[a,b]\times[c,d]} f(x,y) d_qx d_qy
=\int_c^d\int_a^b f(x,y) d_qx d_qy,
\]
and, for $\vec x = (x_1,\dots,x_6)$, 
\begin{multline*}
\int_{a\le x_3\le x_1\le x_5\le x_2\le x_4\le x_6\le b} f(\vec x) d_q\vec x
=\int_{\{(x_1,\dots,x_6):a\le x_3\le x_1\le x_5\le x_2\le x_4\le x_6\le b\}} 
f(\vec x) d_q\vec x\\
=\int_{a}^{b} \int_{a}^{x_6} \int_{x_5}^{x_6}
\int_{a}^{x_5} \int_{x_5}^{x_4} \int_{x_3}^{x_5}
  f(x_1,\dots,x_6) d_qx_1\cdots d_q x_6.
\end{multline*}
\end{exam}

Note that we have
\[
\lim_{q\to1}\int_{D} f(x_1,\dots,x_n) d_qx_{\pi_1}\cdots d_qx_{\pi_n}
=\int_{D} f(x_1,\dots,x_n) dx_{\pi_1}\cdots dx_{\pi_n},
\]
which is independent of the ordering $\pi$ if Fubini's theorem holds.
Unlike usual integrals, we do not always have Fubini's theorem for
$q$-integrals of well behaved functions. For example, 
\[
\int_{0\le x_1\le x_2\le 1} d_qx_1d_qx_2 = \frac{1}{1+q}, \qquad
\int_{0\le x_1\le x_2\le 1} d_qx_2d_qx_1 = \frac{q}{1+q}.
\]

We next define the $q$-volume of a convex polytope. Since the
$q$-integral depends on the order of integration, we need to specify
that ordering. 

\begin{defn}
  Suppose that $D$ is a convex polytope in $\mathbb{R}^n$ and
  $\pi=\pi_1\cdots \pi_n\in S_n$. Then the \emph{$q$-volume of
    $D$ with respect to $\pi$} is
\[
V_q(D,\pi) = \int_D d_q x_{\pi_1}\cdots d_q x_{\pi_n}.
\]
If $\pi=12\cdots n$ is the identity permutation, then we will omit $\pi$
and simply write $V_q(D)=V_q(D,\pi)$, that is,
\[
V_q(D) = \int_D d_q x_1\cdots d_q x_n.
\]
\end{defn}

\begin{exam}
If $D=\{(x_1,x_2,x_3)\in[0,1]^3: x_1\le x_3\le x_2\}$, then
\[
V_q(D,312)= \int_{0\le x_1\le x_3\le x_2\le1} d_qx_3 d_qx_1 d_qx_2,
\]
and
\[
V_q(D)=  \int_{0\le x_1\le x_3\le x_2\le1} d_qx_1 d_qx_2 d_qx_3.
\]
\end{exam}

In most of this paper we will integrate over order polytopes of
partially ordered sets (posets).
\begin{defn}
\label{defn:labeling}
If $P$ is a poset with $n$ elements, a \emph{labeling} of $P$ is a
bijection $\omega:P\to[n]$. A pair $(P,\omega)$ of a poset $P$ and its
labeling $\omega$ is called a \emph{labeled poset}.  If
$\omega(x)\le \omega(y)$ for any $x\le_P y$, we say that $\omega$ is a
\emph{natural labeling} of $P$, or $P$ is \emph{naturally labeled}.
\end{defn}

We need some rudiments of $P$-partition theory, which appear in \cite[Chapter 3]{EC1}.

\begin{defn} Let $(P,\omega)$ be a labeled poset.
A \emph{$(P,\omega)$-partition} is a
function $\sigma:P\to \{0,1,2,\dots\}$ such that 
\begin{itemize}
\item $\sigma(x)\ge \sigma(y)$ if $x\le_P y$,
\item $\sigma(x)>\sigma(y)$ if $x\le_P y$ and $\omega(x)> \omega(y)$.
\end{itemize}
For a $(P,\omega)$-partition $\sigma$, the \emph{size of $\sigma$} is defined by 
\[
|\sigma| = \sum_{x\in P}\sigma(x).
\]
\end{defn}

\begin{defn}
A \emph{linear extension of $P$} is an arrangement
$(t_1,t_2,\dots,t_n)$ of the elements in $P$ such that if $t_i<_P t_j$
then $i<j$.  
\end{defn}
\begin{defn}
The \emph{Jordan-H\"older set} $\LL(P,\omega)$ of $P$ is
the set of permutations of the form
$\omega(t_1)\omega(t_2)\cdots \omega(t_n)$ for some linear extension
$(t_1,t_2,\dots,t_n)$ of $P$.
\end{defn}

It is well known \cite[Theorem~3.15.7]{EC1} that
\begin{equation}
  \label{eq:lin_ext}
\sum_{\sigma} q^{|\sigma|}
=\frac{\sum_{\pi\in\LL(P,\omega)}q^{\maj(\pi)}}{(q;q)_n},
\end{equation}
where the sum is over all $(P,\omega)$-partitions $\sigma$. 

We next define a polytope obtained from a poset in a natural way.  For
simplicity, we will use the same letter $x_i$ for the elements $x_i$
of a poset $P$, the coordinates of $\mathbb R^n$, and also the
integration variables.

\begin{defn}
\label{defn:orderpolytope}
Let $P$ be a poset on $\{x_1,\dots,x_n\}$.  For an
$n$-dimensional box 
\[
I=\{(x_1,\dots,x_n): a_i\le x_i\le b_i\},
\]
\emph{the truncated order polytope of $P$ inside $I$} is defined by
\[
\order_I(P) = \{(x_1,\dots,x_n)\in I: x_i\le x_j \mbox{ if }
x_i\le_P x_j\}.
\]
\emph{The order polytope of $P$} is defined by 
\[
\order(P) = \order_{[0,1]^n}(P).
\]
\end{defn}

An important special case of order polytopes is a simplex, which is an
order polytope of a chain. Let us first define chains and anti-chains.

\begin{defn}
  Let $P$ be a poset. Two elements $x$ and $y$ are called
  \emph{comparable} if $x\le_P y$ or $y\le_P x$, and
  \emph{incomparable} otherwise. A \emph{chain} is a poset in which
  any two elements are comparable. An \emph{anti-chain} is a poset in
  which any two distinct elements are incomparable. 
\end{defn}

\begin{defn}
  For $\pi=\pi_1\cdots\pi_n\in S_n$, we denote by $P_{\pi}$ the chain
  on $\{x_1,\dots,x_n\}$ with relations
  $x_{\pi_1}<x_{\pi_2}<\cdots< x_{\pi_n}$.  For real numbers $a<b$, we
  call
\[
\order_{[a,b]^n}(P_{\pi}) = \{(x_1,\dots,x_n)\in \mathbb{R}^n
: a\le x_{\pi_1} \le \dots\le x_{\pi_n}\le b\}
\]
the \emph{truncated simplex}.
\end{defn}

Note that $\order_{[0,1]^n}(P_{\pi})= \order(P_{\pi})$ is the standard simplex 
which corresponds to a permutation $\pi$ and has volume $1/n!.$

We end this section with definitions for partitions and Schur functions.

\begin{defn}
  A \emph{partition} is a weakly decreasing sequence
  $\lambda=(\lambda_1,\dots,\lambda_n)$ of non-negative integers. Each
  nonzero $\lambda_i$ is called a \emph{part} of $\lambda$. The
  \emph{length} $\ell(\lambda)$ of $\lambda$ is the number of
  parts. We identify a partition $\lambda$ with its \emph{Young
    diagram}
\[
\lambda = \{(i,j): 1\le i\le \ell(\lambda), \quad 1\le j \le \lambda_i\}.
\]
The \emph{transpose} $\lambda'$ of a partition $\lambda$ is defined by
\[
\lambda' = \{(j,i): 1\le i\le \ell(\lambda), \quad 1\le j \le \lambda_i\}.
\]
If $\lambda$ has
  $m_i$ parts equal to $i$ for $i\ge1$, we also write $\lambda$ as
  $(1^{m_1},2^{m_2},\dots)$.  For two partitions
  $\lambda=(\lambda_1,\dots,\lambda_n)$ and $\mu=(\mu_1,\dots,\mu_n)$
  we define
\[
\lambda + \mu = (\lambda_1+\mu_1,\lambda_2+\mu_2,\dots,
\lambda_n+\mu_n).
\]
We also define
\[
\delta_n = (n-1,n-2,\dots,1,0).
\]
\end{defn}

\begin{defn}
  Let $\Par_n$ denote the set of all partitions with length at most
  $n$.  For a partition $\lambda\in\Par_n$, we denote
\[
b(\lambda) = \sum_{i=1}^n (i-1) \lambda_i,
\]
and
\[
q^{\lambda}=(q^{\lambda_1},\dots,q^{\lambda_n}).
\]
\end{defn}

\begin{defn}
For a partition $\lambda=(\lambda_1,\dots,\lambda_n)$, 
the \emph{alternant} $a_\lambda(x_1,\dots,x_n)$ is defined by 
\[
a_\lambda(x_1,\dots,x_n) = 
\det(x_j^{\lambda_i+n-i})_{i,j=1}^n. 
\]
\end{defn}

\begin{defn}
For a partition $\lambda=(\lambda_1,\dots,\lambda_n)$, 
the \emph{Schur function} $s_\lambda(x_1,\dots,x_n)$ is defined by 
\[
s_\lambda(x_1,\dots,x_n) = 
\frac{a_{\lambda+\delta_n}(x_1,\dots,x_n)}{a_{\delta_n}(x_1,\dots,x_n)}. 
\]
\end{defn}

\begin{remark}Note that denominator of the Schur function is the Vandermonde determinant 
$$
a_{\delta_n}(x_1,\dots,x_n)=
\prod_{1\le i<j\le n} (x_i-x_j)=\Delta(x).
$$
We shall also use a version of the Vandermonde determinant which is positive on 
$x_1\le x_2\le \dots \le x_n$,
$$
\DeltaBar(x)= \prod_{1\le i<j\le n} (x_j-x_i)=(-1)^{\binom{n}{2}}\Delta(x).
$$
\end{remark}

\section{Properties of $q$-integrals}
\label{sec:prop-q-integr}

In this section we prove several basic properties of the
$q$-integrals. We give explicit examples when Fubini's theorem of interchanging the 
order of integration is allowed 
(see Proposition~\ref{prop:fubini1},  
Corollary~\ref{cor:sym}), and one example how it fails
(see Proposition~\ref{prop:switch}). 
Finally we give a general technical expansion of a $q$-integral
over a special polytope in Proposition~\ref{firsttechlemma}. It is applied to an arbitrary interchange of the order of integration in Corollary~\ref{cor:change_order}.

\begin{lem}
\label{lem:diff}
If $a$ and $b$ are integers such that $a\le b$, then 
  \[
\int^{q^{a}}_{q^{b}} f(x) d_qx
=(1-q)\sum_{i=a}^{b-1} f(q^i) q^i.
\]
\end{lem}
\begin{proof} This follows easily using the definition \eqref{q-intdefn}.
\end{proof}

When the range for $x$ and $y$ are independent, then we can change the
order of integration for these variables. 

\begin{prop}
\label{prop:fubini1}
We have
\[
\int_{c}^{d}\int_{a}^{b} f(x,y) d_q x d_q y
=\int_{a}^{b}\int_{c}^{d} f(x,y) d_q y d_q x.
\]
\end{prop}
\begin{proof}
  This follows immediately from the definition \eqref{q-intdefn} of the $q$-integral.
\end{proof}

An example of the failure of Fubini's theorem is the following two variable 
computation on triangles. Note that if $q\to 1$, the difference is 0. 

\begin{prop}
\label{prop:switch}
For $a\le b$, we have
\[
\int_{a\le x\le y\le b} f(x,y) d_q x d_q y
-\int_{a\le x\le y\le b} f(x,y)  d_q y d_q x
=(1-q)\int_{a\le x\le b} xf(x,x)  d_q x.
\]
\end{prop}
\begin{proof}
The first integral on the left side is
\begin{multline*}
L_1= (1-q) \int_{a\le y\le b} \sum_{i=0}^\infty 
\bigl( y f(yq^i,y)-a f(aq^i,y)\bigr) q^i d_qy\\
= (1-q)^2 \sum_{i,j=0}^\infty \biggl( \biggl( b^2q^{i+2j}f(bq^{i+j},bq^j)-ab
q^{i+j}f(aq^{i},bq^j)\biggr) \\
- \biggl( a^2q^{i+2j}f(aq^{i+j},aq^j)-a^2
q^{i+j}f(aq^{i},bq^j) \biggr)\biggr).
\end{multline*}

Similarly the second integral on the left side is
\begin{multline*}
L_2= (1-q) \int_{a\le x\le b} \sum_{j=0}^\infty 
\bigl( b f(x,bq^j)-x f(x,xq^j)\bigr) q^j d_qx\\
= (1-q)^2 \sum_{i,j=0}^\infty \biggl(\biggl( b^2q^{i+j}f(bq^{i},bq^j)-b^2
q^{2i+j}f(bq^{i},bq^{i+j})\biggr)\\
- \biggl( abq^{i+j}f(aq^{i},bq^j)-a^2
q^{2i+j}f(aq^{i},bq^{i+j}) \biggr)\biggr).
\end{multline*}

So the difference is
\begin{multline*}
L_1-L_2=(1-q)^2 \sum_{i,j=0}^\infty \bigg( b^2\left(
q^{i+2j} f(bq^{i+j}, bq^j)-q^{i+j} f(bq^{i}, bq^j)+q^{2i+j} f(bq^{i},
bq^{i+j})\right)\\
-a^2\left(
q^{i+2j} f(aq^{i+j}, aq^j)-q^{i+j} f(aq^{i}, aq^j)+q^{2i+j} f(aq^{i},
aq^{i+j})\right) \bigg).
\end{multline*}

Let's check when a term $q^{s+t}f(bq^s,bq^t)$ occurs in the first three terms, 
for non-negative integers $s$ and $t$. 
The first term allows $s\ge t$, the third term $s\le t,$ while the second term is all $s,t.$
So the remaining terms are
$$
(1-q)^2\sum_{s=0}^\infty \left( b^2 q^{2s} f(bq^s,bq^s)-a^2 q^{2s} f(aq^s,aq^s)\right)=
(1-q) \int_{a}^b x f(x,x) d_q x.
$$
\end{proof}

Note that Proposition~\ref{prop:switch} implies that if $f(x,y)$
vanishes on the boundary $x=y$, then we can exchange the order of integration
\[
\int_{a\le x\le y\le b} f(x,y) d_q x d_q y
=\int_{a\le x\le y\le b} f(x,y)  d_q y d_q x.
\]
We show that this is also true for certain polytopes, see
Corollary~\ref{cor:vanish_boundary}. 

Proposition~\ref{firsttechlemma} gives a general expansion 
to evaluate $q$-integrals over special polytopes.

\begin{prop}
\label{firsttechlemma}
Suppose that $S$ is a set of pairs $(i,j)$ with $1\le i\ne j\le n$.
For each element $(i,j)\in S$, let $t_{i,j}$ be an integer.  For fixed
integers $r_i\ge s_i\ge0$, let
\[
Q = \{q^{r_i}\le x_i\le q^{s_i}:1\le i\le n\}
\cap \{ q^{t_{i,j}} x_i\le x_j: (i,j)\in S\}.
\]
Then for $\pi=\pi_1\dots\pi_n\in S_n$ we have
\[
\int_Q f(x_1,\dots,x_n) d_q x_{\pi_1} \cdots d_q x_{\pi_n}
=(1-q)^n\sum_{k_1,\dots,k_n} f(q^{k_1},\dots,q^{k_n})
q^{k_1+\cdots+k_n},
\]
where the sum is over all integers $k_1,\dots,k_n$ satisfying
\begin{itemize}
\item $s_i\le k_i< r_i$ for $1\le i\le n$,
\item $t_{i,j}+k_i\ge k_j$ if $(i,j)\in S$ and $\pi^{-1}(i)<\pi^{-1}(j)$,
\item $t_{i,j}+k_i> k_j$ if $(i,j)\in S$ and $\pi^{-1}(i)>\pi^{-1}(j)$.
\end{itemize}
\end{prop}
\begin{proof}
  Let $D$ be the set of points in $\mathbb{R}^n$ satisfying the
  inequalities in $Q$.  By definition, the left hand side is
  \begin{equation}
    \label{eq:30}
\int_{\min(D_n)}^{\max(D_n)} \cdots \int_{\min(D_1)}^{\max(D_1)}
f(x_1,\dots,x_n) d_q x_{\pi_1} \cdots d_q x_{\pi_n},
  \end{equation}
where 
\[
D_i = D_i(x_{\pi_{i+1}},\dots,x_{\pi_{n}})=\{y_{\pi_i}: (y_1,\dots,y_n)\in D,
y_{\pi_j}=x_{\pi_j} \text{ for } j>i\}.
\]
Note that $D_i$ is the set of all real numbers $y_{\pi_i}$ such that
\begin{itemize}
\item $q^{r_{\pi_i}}\le y_{\pi_i}\le q^{s_{\pi_i}}$,
\item $q^{t_{\pi_i,\pi_j}} y_{\pi_i}\le x_{\pi_j}$ for all $j>i$ with
  $(\pi_i,\pi_j)\in S$,
\item $q^{t_{\pi_j,\pi_i}} x_{\pi_j}\le y_{\pi_i}$ for all $j>i$ with
  $(\pi_j,\pi_i)\in S$. 
\end{itemize}
Thus
\begin{align*}
\max(D_i) &= \min\{q^{s_{\pi_i}},q^{-t_{\pi_i,\pi_j}}x_{\pi_j}:
j>i, (\pi_i,\pi_j)\in S\},\\
\min(D_i) &= \max\{q^{r_{\pi_i}},q^{t_{\pi_j,\pi_i}}x_{\pi_j}:
j>i, (\pi_i,\pi_j)\in S\}.
\end{align*}
In other words, if $x_{\pi_j}=q^{k_{\pi_j}}$ for $j=i+1,\dots,n$, then
$\max(D_i)=q^{a_i}$ and $\min(D_i)=q^{b_i}$, where
\begin{align*}
a_i &= a_i(k_{\pi_{i+1}},\dots,k_{\pi_n})=\max(\{s_{\pi_i}\}\cup
\{ k_{\pi_j}-t_{\pi_i,\pi_j}: j>i, (\pi_i,\pi_j)\in S\}),\\
b_i &= b_i(k_{\pi_{i+1}},\dots,k_{\pi_n})=\min(\{r_{\pi_i}\}\cup
\{ k_{\pi_j}+t_{\pi_j,\pi_i}: j>i, (\pi_j,\pi_i)\in S\}).
\end{align*}
Thus, by Lemma~\ref{lem:diff}, \eqref{eq:30} is equal to
\[
(1-q)^n \sum_{k_{\pi_n}=a_n}^{b_n-1} \cdots \sum_{k_{\pi_1}=a_1}^{b_1-1}
f(q^{k_1}, \dots, q^{k_n})q^{k_1+\cdots+k_n}.
\]
Note that once $k_{\pi_{i+1}},\dots,k_{\pi_n}$ are determined, we have
$a_i\le k_{\pi_i}<b_i$ if and only if
\begin{itemize}
\item $s_{\pi_i}\le k_{\pi_i}<r_{\pi_i}$,
\item $k_{\pi_j}-t_{\pi_i,\pi_j}\le k_{\pi_i}$ for all $j>i$ with
    $(\pi_i,\pi_j)\in S$,
\item $k_{\pi_i} < k_{\pi_j}+t_{\pi_j,\pi_i}$ for all $j>i$ with
    $(\pi_j,\pi_i)\in S$.
\end{itemize}
One can easily check that the above sum is equivalent to the right
hand side of the equation in this proposition.
\end{proof}

An immediate corollary of Proposition~\ref{firsttechlemma} is that if
$f(x_1,\dots,x_n)$ vanishes on the boundary $q^{t_{i,j}}x_i=x_j$ for
all $(i,j)\in S$, then we can change the order of integration.

\begin{cor}
\label{cor:vanish_boundary}
We follow the same notation in Proposition~\ref{firsttechlemma}.
Suppose that the function $f$ satisfies that $f(x_1,\dots,x_n)=0$ if
$q^{t_{i,j}}x_i=x_j$ for any pair $(i,j)\in S$. Then
for any permutations $\pi,\sigma\in S_n$, we have
\[
\int_Q f(x_1,\dots,x_n) d_q x_{\pi_1} \cdots d_q x_{\pi_n}
=\int_Q f(x_1,\dots,x_n) d_q x_{\sigma_1} \cdots d_q x_{\sigma_n}.
\]
\end{cor}

Corollary~\ref{cor:sym} gives a sufficient condition for changing the
domain of a $q$-integral from one standard simplex to another. We will
use this corollary later in this paper. 

\begin{cor}
\label{cor:sym}
  Suppose that $f(x_1,\dots,x_n)$ is symmetric in $x_1,\dots,x_n$ and
  $f(x_1,\dots,x_n) = 0$ if $x_i=x_j$ for any $i\ne j$. Then for
  $\pi\in S_n$ and $0< q< 1$, we have
\[
\int_{0\le x_{\pi_1} \le \dots\le x_{\pi_n}\le 1} f(x_1,\dots,x_n) \dqx = 
\int_{0\le x_1\le \cdots \le x_n \le 1} f(x_1,\dots,x_n) \dqx.
\]  
\end{cor}
\begin{proof}
  Let $\sigma=\pi^{-1}$.  By renaming the variables
  $x_i \mapsto x_{\sigma_i}$, the left hand side becomes
\[
\int_{0\le x_{1} \le \dots\le x_{n}\le 1} f(x_{\sigma_1},\dots,x_{\sigma_n})
d_qx_{\sigma_1} \cdots d_qx_{\sigma_n} .
\]
Since $f$ is symmetric, this is equal to
\[
\int_{0\le x_{1} \le \dots\le x_{n}\le 1} f(x_1,\dots,x_n)
d_qx_{\sigma_1} \cdots d_qx_{\sigma_n} .
\]
We finish the proof by applying Corollary~\ref{cor:vanish_boundary}.
\end{proof}

Corollary~\ref{cor:change_order} allows us to change the order of
integration by modifying the polytope. 

\begin{cor}
\label{cor:change_order}
Let $Q$ be a family of inequalities of the form
$q^{t_{i,j}} x_i\le x_j$ for $1\le i\ne j\le n$ and the inequalities
$q^{r_i}\le x_i\le q^{s_i}$ for $1\le i\le n$. Then for
$\pi,\sigma\in S_n$ we have
\[
\int_Q f(x_1,\dots,x_n) d_q x_{\pi_1} \cdots d_q x_{\pi_n}
=\int_{Q'} f(x_1,\dots,x_n) d_q x_{\sigma_1} \cdots d_q x_{\sigma_n},
\]
where $Q'$ is the family of inequalities obtained from $Q$ as follows.
\begin{itemize}
\item If $Q$ contains the inequality $q^{t_{i,j}} x_i\le x_j$
such that $\pi^{-1}(i)>\pi^{-1}(j)$ and
$\sigma^{-1}(i)<\sigma^{-1}(j)$, then
replace this inequality by $q^{t_{i,j}-1} x_i\le x_j$.
\item If $Q$ contains the inequality $q^{t_{i,j}} x_i\le x_j$
such that $\pi^{-1}(i)<\pi^{-1}(j)$ and
$\sigma^{-1}(i)>\sigma^{-1}(j)$, then
replace this inequality by $q^{t_{i,j}+1} x_i\le x_j$.
\item The remaining inequalities of $Q$ are unchanged. 
\end{itemize}
\end{cor}
\begin{proof}
  This is proved by expanding both sides using
  Proposition~\ref{firsttechlemma}.
\end{proof}

\section{$q$-integrals over order polytopes}
\label{sec:q-integrals-over}

In this section we consider $q$-integrals over order polytopes of posets. 
The main result (Theorem~\ref{thm:pw}) is that the $q$-volume of an order polytope of a poset 
may be written, up to a factor, as the $maj$-generating function of the linear extensions of that poset.
Thus these $q$-integrals may be evaluated using permutation enumeration. 

Recall the truncated order polytope $\order_I(P)$ in
Definition~\ref{defn:orderpolytope}. We first use
Proposition~\ref{firsttechlemma} to evaluate an arbitrary integral
over $\order_I(P)$ as a sum over restricted $(P,\omega)$-partitions $\sigma$.

\begin{thm}
\label{thm:order}
Let $P$ be a poset on $\{x_1,\dots,x_n\}$ and $\omega_n:P\to[n]$ be the
labeling of $P$ given by $\omega_n(x_i)=i$ for $1\le i\le n$.  For
integers $r_1,r_2,\dots,r_n, s_1,s_2,\dots,s_n$ with $r_i\ge s_i\ge0$,
let
\[
I = \{(x_1,\dots,x_n): q^{r_i}\le x_i \le q^{s_i}\}.
\]
Then
\[
\int_{\order_I(P)} f(x_1,\dots,x_n) \dqx=
(1-q)^n \sum_{\sigma}
f(q^{\sigma(x_1)},\dots,q^{\sigma(x_n)}) q^{|\sigma|},
\]
where the sum is over all $(P,\omega_n)$-partitions $\sigma$ 
satisfying $s_i\le\sigma(x_i)<r_i$ for $1\le i\le n$. 
\end{thm}
\begin{proof}
This is obtained immediately from Proposition~\ref{firsttechlemma} by taking
\[
\pi=12\cdots n,\qquad S=\{(i,j): x_i\le_P x_j\},\qquad
t_{i,j}=0.
\]
\end{proof}

Note that in Theorem~\ref{thm:order} the labeling of the poset is
closely related to the order of integration.

The next corollary expresses a $q$-integral over a truncated order
polytope as a sum of $q$-integrals of truncated simplices.

\begin{cor}\label{cor:LL(P)}
Let $P$ be a poset on $\{x_1,\dots,x_n\}$ and $\omega_n:P\to[n]$ the
labeling of $P$ given by $\omega_n(x_i)=i$ for $1\le i\le n$. For
integers $r_1,r_2,\dots,r_n, s_1,s_2,\dots,s_n$ with $r_i\ge s_i\ge0$,
let
\[
I = \{(x_1,\dots,x_n): q^{r_i}\le x_i \le q^{s_i}\}.
\] 
Then
\[
\int_{\order_I(P)} f(x_1,\dots,x_n) \dqx=
\sum_{\pi\in\LL(P,\omega_n)}
\int_{\order_I(P_\pi)} f(x_1,\dots,x_n) \dqx,
\]
$P_\pi$ is the chain $x_{\pi_1} \le \cdots \le x_{\pi_n}$.
\end{cor}
\begin{proof}
  This can be proved by the standard arguments in the
  $(P,\omega)$-partition theory, see \cite[Lemma~3.15.3]{EC1}.
\end{proof}

We shall need the following lemma later.

\begin{lem}\label{lem:par}
Let $f(x_1,\dots,x_n)$ be a function such that $f(x_1,\dots,x_n)=0$ if
$x_i=x_j$ for any $i\ne j$. Then
\[
\sum_{\mu\in\Par_n} q^{|\mu+\delta_n|} f(q^{\mu+\delta_n})
=\frac{1}{(1-q)^n} \int_{0\le x_1\le\cdots\le x_n\le1}
f(x_1,\dots,x_n)\dqx.
\]
\end{lem}
\begin{proof}
  By Theorem~\ref{thm:order} and the assumption on the function
  $f(x_1,\dots,x_n)$, the right side is equal to
\[
\sum_{i_1\ge i_2\ge \cdots \ge i_n\ge0}
f(q^{i_1},\dots,q^{i_n}) q^{i_1+\cdots+i_n}
=\sum_{i_1> i_2> \cdots> i_n\ge0}
f(q^{i_1},\dots,q^{i_n}) q^{i_1+\cdots+i_n}.
\]
This is equal to the left side. 
\end{proof}

By taking $I=[0,1]^n$ and $f(x_1,\dots,x_n)=1$ in
Theorem~\ref{thm:order}, we obtain that the $q$-volume of the order
polytope $\order(P)$ is the generating function for
$(P,\omega)$-partitions. The equivalence of the two equations in
Theorem~\ref{thm:pw} follows from the well known fact \eqref{eq:lin_ext}
on $P$-partition theory. 

\begin{thm}[$q$-volume of order polytope]
\label{thm:pw}
  Let $P$ be a poset on $\{x_1,\dots,x_n\}$ with labeling $\omega_n$
  given by $\omega_n(x_i)=i$.  Then
\[
V_q(\order(P)) = \int_{\order(P)} \dqx
= (1-q)^n \sum_{\sigma} q^{|\sigma|},
\]
where the sum is over all $(P,\omega_n)$-partitions $\sigma$.
Equivalently, 
\[
V_q(\order(P)) =\int_{\order(P)} d_qx_1\cdots d_qx_n
=\frac{1}{[n]_q!} \sum_{\pi\in\LL(P,\omega_n)}q^{\maj(\pi)}.
\]
\end{thm}

As a special case of Theorem~\ref{thm:pw}, let us consider the
anti-chain $P$ on $\{1,2,\dots,n\}$. Then $\order(P)$ is the
$n$-dimensional unit cube whose $q$-volume is $1$ and
$\LL(P,\omega)=S_n$. If we apply Theorem~\ref{thm:pw} to $P$,
we obtain
\[
1=\frac{1}{[n]_q!} \sum_{\pi\in S_n}q^{\maj(\pi)}.
\]
This is a well known result for the maj-generating function for
permutations, see \cite{EC1}.

Let's consider another special case of order polytopes, which are
truncated simplices.  The following lemma will be used to evaluate the
$q$-volume of a truncated simplex.

\begin{lem}\label{lem:lambda}
Let $\pi\in S_n$ and $r>s\ge0$. Then 
\[
\sum_{\substack{r>i_1\ge \cdots\ge i_{n}\ge s\\
i_j>i_{j+1}\ \mathrm{if}\  j\in\Des(\pi)}} q^{i_1+\cdots+i_n}
= q^{sn+\maj(\pi)} \frac{(q^{r-s-\des(\pi)};q)_n}{(q;q)_n}.
\]
\end{lem}
\begin{proof}
  Let $A$ be the set of partitions
  $\lambda=(\lambda_1,\dots,\lambda_n)$ such that
  $r>\lambda_1\ge \cdots\ge \lambda_n\ge s$ and
  $\lambda_j>\lambda_{j+1}$ if $j\in\Des(\pi)$.  Observe that if
  $\lambda\in A$, then considering the Young diagram of the transpose
  $\lambda'$ of $\lambda$ we have
  $(n^s)\subseteq \lambda' \subseteq (n^{r-1})$ and $j\in \lambda'$
  for $j\in\Des(\pi)$. Thus the left hand side is equal to
\[
\sum_{\lambda\in A} q^{|\lambda|} = q^{sn}q^{\maj(\pi)}
\qbinom{r-s-1-\des(\pi)+n}{n} = q^{sn}q^{\maj(\pi)}
\frac{(q^{r-s-\des(\pi)};q)_n}{(q;q)_n},
\]
which finishes the proof.
\end{proof}

We now have a formula for the $q$-volume of a truncated simplex. 
This will be used later to evaluate the $q$-beta integral and give a
combinatorial interpretation for the $q$-Selberg integral. 

\begin{cor} [$q$-volume of a truncated simplex]
\label{cor:ab}
  For $\pi\in S_n$ and real numbers $a<b$, the $q$-volume of the
  truncated simplex $\order_{[a,b]^n}(P_\pi)$ is
\[
V_q(\order_{[a,b]^n}(P_\pi))=\int_{a\le x_{\pi_1} \le \dots\le x_{\pi_n}\le b} \dqx =
\frac{b^n q^{\maj(\pi)}}{[n]_q!} (aq^{-\des(\pi)}/b;q)_{n}.
\]  
\end{cor}
\begin{proof}
Since both sides are polynomials in $a$ and $b$, it is sufficient to
show the following for  integers $r>s\ge0$:
\[
\int_{q^r\le x_{\pi_1} \le \dots\le x_{\pi_n}\le q^s} \dqx =
\frac{q^{sn+\maj(\pi)}}{[n]_q!} (q^{r-s-\des(\pi)};q)_{n}.
\]  
This follows from Theorem~\ref{thm:order} and Lemma~\ref{lem:lambda}. 
\end{proof}

By considering the $q$-volume of the box $[a,1]^n$, we obtain an
identity for a generating function for permutations with $\maj$ and
$\des$ statistics. 

\begin{cor}
For a non-negative integer $n$, we have
\[
\sum_{\pi\in S_n} q^{\maj(\pi)}
(aq^{-\des(\pi)};q)_n = (1-a)^n[n]_q!.
\]  
\end{cor}
\begin{proof}
  Let $I=[a,1]^n$ and $P$ the anti-chain on $\{x_1,\dots,x_n\}$.  Then
  $\order_I(P)=I$ and $V_q(\order_I(P)) =(1-a)^n$. On the other hand,
  by Corollary~\ref{cor:LL(P)} and Corollary~\ref{cor:ab}, we have
\[
(1-a)^n = V_q(\order_I(P)) = \sum_{\pi\in S_n} V_q(\order_{[a,1]^n}(P_\pi))
=\sum_{\pi\in S_n} \frac{q^{\maj(\pi)}}{[n]_q!} (aq^{-\des(\pi)};q)_n.
\]  
By multiplying both sides by $[n]_q!$, we have the stated result.
\end{proof}

Theorem~\ref{thm:carlitz} is another generating function for
permutations with $\maj$ and $\des$ statistics due to MacMahon.  This
result is often called Carlitz's formula \cite{Carlitz1975}, see
\cite[p.~6]{Gessel2015}.

\begin{thm}
\label{thm:carlitz}
We have
\[
\sum_{\pi\in S_n} t^{\des(\pi)} q^{\maj(\pi)} 
= (t;q)_{n+1}\sum_{i\ge0} [i+1]_q^n t^i.
\]
\end{thm}

We show that Corollary~\ref{cor:ab} and Theorem~\ref{thm:carlitz} are
equivalent.  By expanding $(aq^{-\des(\pi)};q)_n$ using the
$q$-binomial theorem and comparing the coefficients of $a^k$ in both
sides, one can restate Corollary~\ref{cor:ab} as follows: for
$0\le k\le n$,
  \begin{equation}
    \label{eq:majdes1}
\sum_{\pi\in S_n} (q^{-k})^{\des(\pi)} q^{\maj(\pi)}
=q^{-\binom k2} \binom nk (q;q)_k (q;q)_{n-k}.
  \end{equation}
  On the other hand, by expanding the numerator of
  $[i+1]_q^n = (1-q^{i+1})^n/(1-q)^n$, using the binomial theorem and
the geometric series, one can check that
  Theorem~\ref{thm:carlitz} can be restated as
  \begin{equation}
    \label{eq:majdes2}
\sum_{\pi\in S_n} t^{\des(\pi)} q^{\maj(\pi)}
=\frac{1}{(1-q)^n} \sum_{i=0}^n \binom ni (-q)^i
(t;q)_i (q^{i+1}t;q)_{n-i}.
  \end{equation}
  If we substitute $t=q^{-k}$ in \eqref{eq:majdes2} for $0\le k\le n$ 
  then we obtain \eqref{eq:majdes1}. In order to obtain
  \eqref{eq:majdes2} from \eqref{eq:majdes1} one can argue as
  follows. By \eqref{eq:majdes1}, we know that \eqref{eq:majdes2} is
  true when $t=q^{-k}$ for $0\le k\le n$.  Since both sides of
  \eqref{eq:majdes2} are polynomials in $t$ of degree at most $n$ and
  the equation has $n+1$ different solutions, both sides are the same
  as polynomials in $t$.

When $a=0$ and $b=1$ in Corollary~\ref{cor:ab} we obtain the following
corollary. It explicitly demonstrates the failure of Fubini's theorem by
finding the differences in the $q$-volumes of the standard $n!$ simplices which 
lie inside an $n$-dimensional cube.

\begin{cor}
[$q$-volume of a simplex]
\label{cor:maj}
For $\pi\in S_n$, the $q$-volume of the simplex $\order(P_\pi)$ is 
\[
V_q(\order(P_\pi))=\int_{0\le x_{\pi_1} \le \dots\le x_{\pi_n}\le 1} \dqx = \frac{q^{\maj(\pi)}}{[n]_q!}.
\]  
\end{cor}

\section{Operations on posets}
\label{sec:operations-posets}

 Suppose that a poset $P$ is modified to obtain another poset $P'.$
Is $V_q(\order(P'))$ related to the $q$-integral defining $V_q(\order(P))$?
In this section we answer this question for the following types of modifications:
\begin{enumerate}
\item attaching a chain below an element (Lemma~\ref{lem:meancat}),
\item attaching a chain above an element (Lemma~\ref{lem:happycat}),
\item attaching a chain between two elements (Lemma~\ref{lem:xyx}),
\item inserting a chain which interlaces another chain (Lemmas~\ref{lem:interlacing1} 
and \ref{lem:interlacing2}).
\end{enumerate}

\begin{figure}
  \centering
\includegraphics{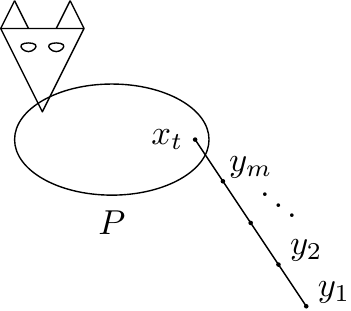} \qquad
\includegraphics{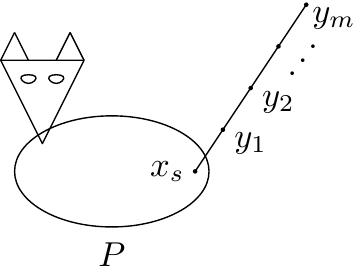}\qquad
\includegraphics{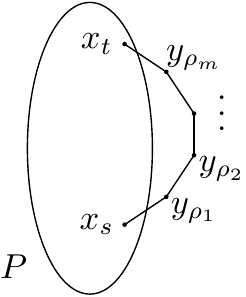}
\caption{Illustrations of the scaredy cat lemma (left), the happy cat
  lemma (middle), and the attaching chain lemma (right).}
\label{fig:catlemma}
\end{figure}

First we insert a chain below a fixed element. See the left figure in
Figure~\ref{fig:catlemma}.
\begin{lem}[Scaredy Cat Lemma]
\label{lem:meancat}
  Let $P$ be a poset on $\{x_1,\dots,x_n\}$ and $t\in[n]$.  Let $Q$ be
  the poset on $\{x_1,\dots,x_n,y_1,\dots,y_m\}$ with relations
  $x_i\le_{Q} x_j$ if and only if $x_i\le_{P} x_j$, and
  $y_{1} \le_{Q} \cdots \le_{Q} y_{m}\le_{Q} x_t$. In other words, $Q$
  is obtained from $P$ by attaching a chain $y_1\le \cdots\le y_m$
  below $x_t$.  Then we have
\[
\int_{\order(P)} x_t^m  f(x_1,\dots,x_n) \dqx 
=[m]_q!\int_{\order(Q)} f(x_1,\dots,x_n)\DQ{y}{m} \dqx.
\]
\end{lem}
\begin{proof}
By Corollary~\ref{cor:ab}, we have
\[
\int_{0\le y_1\le \cdots \le y_m\le x_t} d_q y_1\cdots d_q y_m=
\frac{x_t^m}{[m]_q!}.
\]  
Thus, the left hand side of the equation can be written as
\[
[m]_q!\int_{\order(P)} \left(
\int_{0\le y_1\le \cdots \le y_m\le x_t} d_q y_1\cdots d_q y_m\right)
f(x_1,\dots,x_n) \dqx.
\]
This is equal to the right hand side of the equation. 
\end{proof}

Next we insert a chain above a fixed element. See the middle figure in Figure~\ref{fig:catlemma}.
\begin{lem}[Happy Cat Lemma]
\label{lem:happycat}
  Let $P$ be a poset on $\{x_1,\dots,x_n\}$ and $s\in[n]$.  Let $Q$ be
  the poset on $\{x_1,\dots,x_n,y_1,\dots,y_m\}$ with relations
  $x_i\le_{Q} x_j$ if and only if $x_i\le_{P} x_j$, and
  $\le_Q x_s\le y_{1} \le_{Q} \cdots \le_{Q} y_{m}$. In other words, $Q$
  is obtained from $P$ by attaching a chain $y_1\le \cdots\le y_m$
  above $x_s$.  Then we have
\[
\int_{\order(P)} (qx_s;q)_m f(x_1,\dots,x_n) \dqx 
=[m]_q!\int_{\order(Q)} f(x_1,\dots,x_n)\dqx \DQ{y}{m} .
\]
\end{lem}
\begin{proof}
By Corollary~\ref{cor:ab}, we have
\[
\int_{qx_s\le y_1\le \cdots \le y_m\le 1} d_q y_1\cdots d_q y_m=
\frac{(qx_s;q)_m}{[m]_q!}.
\]  
Thus, the left hand side of the equation can be written as
\begin{align*}
&[m]_q!\int_{\order(P)} \left(
\int_{qx_s\le y_1\le \cdots \le y_m\le 1} d_q y_1\cdots d_q y_m\right)
f(x_1,\dots,x_n) \dqx\\
=& [m]_q!\int_{D} f(x_1,\dots,x_n)d_q y_1\cdots d_q y_m\dqx,
\end{align*}
where $D$ is the set of inequalities given by
\[
D = \{0\le x_i\le x_j\le 1: x_i\le_P x_j\} 
\cup\{0\le y_1\le \cdots\le y_m\le 1\}
\cup \{qx_s\le y_i : i\in [m]\}.
\]
By Corollary~\ref{cor:change_order}, we have
\[
\int_{D} f(x_1,\dots,x_n)d_q y_1\cdots d_q y_m\dqx
=
\int_{D'} f(x_1,\dots,x_n)\dqx d_q y_1\cdots d_q y_m,
\]
where
\[
D' = \{0\le x_i\le x_j\le 1: x_i\le_P x_j\} 
\cup\{0\le y_1\le \cdots\le y_m\le 1\}
\cup \{x_s\le y_i : i\in [m]\}.
\]
Since $D'$ and $\order(Q)$ represent the same domain, we get the
lemma. 
\end{proof}

Now we state but do not prove the attaching chain lemma, since the
proof is similar to the previous proofs. 
See the right figure in Figure~\ref{fig:catlemma}.

\begin{lem} [Attaching chain lemma]
\label{lem:xyx}
  Let $\rho\in S_m$.  Let $P$ be a poset on $\{x_1,\dots,x_n\}$ with
  $x_s \le_P x_t$.  Define $Q$ to be the poset on
  $\{x_1,\dots,x_n,y_1,\dots,y_m\}$ with relations
  $x_i\le_{Q} x_j$ if and only if $x_i\le_{P} x_j$, and
  $x_{s} \le_{Q} y_{\rho_1} \le_{Q} \cdots \le_{Q}
  y_{\rho_m}\le_{Q} x_{t}$.  Then, we have
  \begin{multline}
    \label{eq:xyx}
\int_{\order(P)} 
q^{\maj(\rho)} x_t^m (q^{-\des(\rho)} x_{s}/x_{t};q)_m
f(x_1,\dots,x_n)\dqx\\
=[m]_q!\int_{\order(Q)} 
f(x_1,\dots,x_n) \DQ{y}{m} \dqx,
  \end{multline}
  \begin{multline}
    \label{eq:xyx''}
\int_{\order(P)} 
q^{\maj(\rho)} x_t^m (q^{1-\des(\rho)} x_{s}/x_{t};q)_m
f(x_1,\dots,x_n)\dqx\\
=[m]_q!\int_{\order(Q)} 
f(x_1,\dots,x_n) 
d_qx_1\cdots d_q x_s d_qy_1\cdots d_q y_m
d_qx_{s+1}\cdots d_q x_n.
  \end{multline}
\end{lem}

We note that \eqref{eq:xyx''} holds if the order of integration on the
right hand side is obtained from $d_qx_1\cdots d_q x_n$ by inserting
$d_qy_1\cdots d_q y_m$ anywhere between $x_s$ and $x_t$.

There is another way of attaching a chain. We need a
definition which is often called the interlacing condition. 

\begin{defn}
For two sets of variables $x=(x_1,x_2,\dots,x_{n-1})$ and
$y=(y_1,y_2,\dots,y_{n})$ with $y$ having one more variable than $x$,
we denote by $x\prec y$ the set of inequalities
\[
\{ y_1 \le x_1 \le y_2 \le x_2\le \cdots \le y_{n-1}
\le x_{n-1} \le y_{n}\}.
\]
Moreover, if $x=(x_1,x_2,\dots,x_{n})$ and $y=(y_1,y_2,\dots,y_{n})$
have the same number of variables, we also denote $x\prec y$ the
set of inequalities
\[
\{  x_1 \le y_1 \le x_2 \le y_2\le \cdots 
\le x_{n} \le y_{n}\}.
\]
See Figure~\ref{fig:x<y}. 
\end{defn}

\begin{figure}
  \centering
\includegraphics{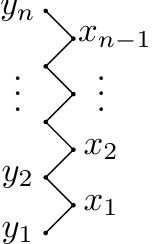} \qquad \qquad \qquad
\includegraphics{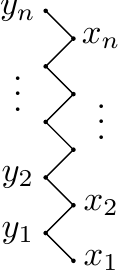} 
\caption{Illustrations of $x\prec y$ for $x=(x_1,\dots,x_{n-1})$,
  $y=(y_1,\dots,y_n)$ on the left and $x\prec y$ for
  $x=(x_1,\dots,x_{n})$, $y=(y_1,\dots,y_n)$ on the right.}
\label{fig:x<y}
\end{figure}

\begin{lem}[Interlacing chain lemma 1]
\label{lem:interlacing1}
Let $P$ be a poset on $\{x_1,\dots,x_N\}$. Suppose that $y=(y_1<_P
\cdots <_P y_n)$ is a chain in $P$. Let $Q$ be the poset obtained
from $P$ by adding a new chain $z=(z_1< \cdots < z_{n-1})$ which
interlaces with $y$ as follows:
\[
y_1<_Q z_1<_Q y_2<_Q z_2 <_Q\cdots <_Q y_{n-1} <_Q z_{n-1} <_Q y_n.
\]
Then, for a partition $\lambda$ with $\ell(\lambda)< n$, we have
\begin{multline*}
\int_{\order(P)}
s_{\lambda}(y)\DeltaBar(y) f(x) d_qx_1\cdots d_q x_{N}
\\=\prod_{i=1}^{n-1} [\lambda_i+n-i]_q
\int_{\order(Q)} s_{\lambda}(z)\DeltaBar(z)  f(x)
d_qz_1\cdots d_q z_{n-1} d_q x_1 \cdots d_q x_{N}. 
\end{multline*}
\end{lem}
\begin{proof}
This follows from Lemma~\ref{lem:schur1} below.  
\end{proof}

\begin{lem}[Interlacing chain lemma 2]
\label{lem:interlacing2}
Let $P$ be a poset on $\{x_1,\dots,x_N\}$. Suppose that $y=(y_1<_P
\cdots <_P y_n)$ is a chain in $P$. Let $Q$ be the poset obtained
from $P$ by adding a new chain $z=(z_1< \cdots < z_{n})$ which
interlaces with $y$  as follows:
\[
z_1<_Q y_1 <_Q z_2<_Q y_2 <_Q\cdots <_Q z_{n} <_Q y_n.
\]
Then, for a partition $\lambda$ with $\ell(\lambda) = n$, we have
\begin{multline*}
\int_{\order(P)}
s_{\lambda}(y)\DeltaBar(y)  f(x)
d_qx_1\cdots d_q x_{N}
\\=\prod_{i=1}^{n} [\lambda_i+n-i]_q
\int_{\order(Q)} s_{\lambda-(1^n)}(z)\DeltaBar(z)
 f(x) d_qz_1\cdots d_q z_{n} d_q x_1 \cdots d_q x_{N}. 
\end{multline*}
\end{lem}
\begin{proof}
This follows from Lemma~\ref{lem:schur2} below.  
\end{proof}

The following lemma is a special case of 
\cite[Theorem 1]{Okounkov1998}, which has Macdonald polynomials
instead of Schur functions. 

\begin{lem}
\label{lem:schur1}
  Let $z=(z_1,z_2,\dots,z_{n-1})$ and $y=(y_1,y_2,\dots,y_{n})$.
If $\ell(\lambda)< n$,  we have
  \[
s_\lambda(y) \DeltaBar(y)
=\prod_{i=1}^{n-1} [\lambda_i+n-i]_q
\int_{z\prec y} s_\lambda(z) \DeltaBar(z) 
d_q z_1 \cdots d_q z_{n-1}. 
\]
\end{lem}
\begin{proof} We note that the integrand on the right side is a 
determinant. Applying the $q$-integrals gives another determinant, which is the left side. Specifically,
$$
s_\lambda(z_1,\dots,z_{n-1}) \DeltaBar(z)=(-1)^{\binom{n-1}{2}}
\det\left( z_i^{\lambda_j+n-1-j}\right)_{1\le i,j \le n-1}.
$$
Evaluating the $q$-integral on $z_i$ from $y_{i}$ to $y_{i+1}$ evaluates the 
right side of Lemma~\ref{lem:schur1} as 
$$
(-1)^{\binom{n-1}{2}} 
\det\left( y_{i+1}^{\lambda_j+n-j}-y_{i}^{\lambda_j+n-j}\right)_{1\le i,j \le n-1}.
$$

For the left side, 
$$
s_\lambda(y_1,\dots,y_{n}) \DeltaBar(y)
= (-1)^{\binom{n}{2}} \det\left( y_{i}^{\lambda_j+n-j}\right)_{1\le i,j \le n}
$$
has an $n^{th}$ column of all $1$'s because $\lambda_n=0$. 
By successively subtracting 
the $i^{th}$ row from the $i+1^{th}$ row for $i=n-1$ to $i=1$ 
we obtain a final column of $(1,0,\dots,0)^T$. Expanding the 
determinant along this column gives 
$$
\begin{aligned}
s_\lambda(y_1,\dots,y_{n}) \DeltaBar(y)=&
(-1)^{\binom{n}{2}}(-1)^{n-1}
\det\left( y_{i+1}^{\lambda_j+n-j}-y_{i}^{\lambda_j+n-j}\right)_{1\le i,j \le n-1}\\
=& (-1)^{\binom{n-1}{2}}
\det\left( y_{i+1}^{\lambda_j+n-j}-y_{i}^{\lambda_j+n-j}\right)_{1\le i,j \le n-1}
\end{aligned}
$$
which is the right side.
\end{proof}

The version of Lemma~\ref{lem:schur1} when $\lambda_n>0$ is next.

\begin{lem}
\label{lem:schur2}
  Let $z=(z_1,z_2,\dots,z_{n})$ and $y=(y_1,y_2,\dots,y_{n})$.
If $\ell(\lambda) = n$,  we have
  \[
s_\lambda(y) \DeltaBar(y)
=\prod_{i=1}^{n} [\lambda_i+n-i]_q
\int_{z\prec y} s_{\lambda-(1^n)}(z) \DeltaBar(z) 
d_q z_1 \cdots d_q z_{n}. 
\]
\end{lem}
\begin{proof}
  Let $u = (u_1,u_2,\dots,u_{n+1})=(0,y_1,y_2,\dots,y_n)$.  By
  Lemma~\ref{lem:schur1}, we have
  \[
s_{\lambda-(1^n)}(u_1,\dots,u_{n+1}) \DeltaBar(u)
=\prod_{i=1}^{n} [(\lambda_i-1)+(n+1)-i]_q
\int_{z\prec u} s_{\lambda-(1^n)}(z_1,\dots,z_{n}) \DeltaBar(z) 
d_q z_1 \cdots d_q z_{n}. 
\]
Since 
\[
s_{\lambda-(1^n)}(u_1,\dots,u_{n+1}) 
=s_{\lambda-(1^n)}(y_1,\dots,y_{n})
=s_{\lambda}(y_1,\dots,y_{n})/y_1\cdots y_n
\]
and
\[
\DeltaBar(u) = y_1\cdots y_n \DeltaBar(y)
\]
we are done.
\end{proof}

\section{Examples of $q$-integrals}
\label{sec:examples}

In this section we use the constructions in
Section~\ref{sec:q-integrals-over} to evaluate the $q$-integrals.
This includes the $q$-beta integral, a $q$-analogue of Dirichlet's
integral, and a general $q$-beta integral due to Andrews and Askey
\cite{AndrewsAskey1981}.  We will then find a connection with linear
extensions of forest posets.
 
\subsection{The $q$-beta integral}

The following is the well known integral called the $q$-beta integral.
We now prove this using our methods. The idea is to add two chains to
a point: one chain below it and the other chain above it, naturally labeled. 
The resulting poset is again a chain whose $q$-volume is easily computed. 

\begin{cor}
\label{cor:firstqbeta}
We have
\[
\int_{0}^1 x^n (xq;q)_m d_qx  =\frac{[n]_q![m]_q!}{[n+m+1]_q!}.
\]  
\end{cor}
\begin{proof}
By \eqref{lem:happycat} and \eqref{lem:meancat}  we
have
\[
\int_{0}^1 x^n (xq;q)_m d_qx  =
[n]_q![m]_q! 
\int_{0\le y_1\le\cdots\le y_n\le x\le z_1\le\cdots\le z_m\le1}
d_q y_1 \cdots d_q y_n d_qx d_q z_1 \cdots d_q z_m.
\]
By Corollary~\ref{cor:maj}, we get the $q$-beta integral formula.
\end{proof}

\subsection{A $q$-analog of Dirichlet integral}

We now consider the simplex
\[
\Omega_n=\{(x_1,\dots,x_n)\in[0,1]^n: x_1+\cdots+x_n\le 1\}.
\]

Dirichlet's integral is the following, see
\cite[Theorem~1.8.6]{AAR_SP}:
\begin{equation}
  \label{eq:dirichlet}
\int_{\Omega_n} x_1^{\alpha_1-1}\cdots x_n^{\alpha_n-1} (1-x_1-\cdots -x_n)^{a_{n+1}-1}
dx_1\cdots dx_n
=\frac{\Gamma(\alpha_1)\cdots\Gamma(\alpha_{n+1})}
{\Gamma(\alpha_1+\dots+\alpha_{n+1})}.
\end{equation}

By introducing new variables $y_i=x_1+\cdots+x_i$ and integers
$k_i=\alpha_i-1$, we get an equivalent version of \eqref{eq:dirichlet}
\begin{multline}
  \label{eq:dirichlet'}
\int_{0\le y_1\le y_2\le \cdots\le y_n\le 1} y_1^{k_1}(y_2-y_1)^{k_2}\cdots 
(y_n-y_{n-1})^{k_n} (1-y_n)^{k_{n+1}}
dy_1\cdots dy_n\\
=\frac{k_1!\cdots k_{n+1}!}{(n+k_1+\cdots+k_{n+1})!}.
\end{multline}

A $q$-analogue of \eqref{eq:dirichlet'} is given by the next corollary.
It generalizes Corollary~\ref{cor:firstqbeta} which is the case $n=1$. 
  
\begin{cor}\label{cor:dirichlet2}
For nonnegative integers $k_1,\dots,k_{n+1}$, we have
\[
\int_{0\le y_1\le y_2\le \cdots\le y_{n}\le 1} 
y_1^{k_1}(qy_n;q)_{k_{n+1}}
\prod_{i=2}^n y_i^{k_i}(qy_{i-1}/y_i;q)_{k_i}
d_qy_1\cdots d_qy_{n}=
\frac{[k_1]_q!\cdots [k_{n+1}]_q!}{[n+k_1+\cdots+k_{n+1}]_q!}.
\]  
\end{cor}

\begin{proof}
We generalize the proof of Corollary~\ref{cor:firstqbeta}. 
We start with a chain $y_1\le y_2\le \cdots\le y_n.$
Attach a chain with $k_i$ elements between
$y_{i-1}$ and $y_i$, for $2\le i\le n$, and also a chain with 
$k_1$ elements below $y_1$, and a chain with 
$k_{n+1}$ elements above $y_n$. We obtain a chain with 
$n+k_1+\cdots+k_{n+1}$
elements which is naturally labeled.

By applying the scaredy cat lemma once, the second part \eqref{eq:xyx''}
in the attaching chain lemma with $\rho$ the identity permutation
$n-1$ times, and the happy cat lemma once, we obtain the left side
multiplied by some $q$-factorials. The right side results from
Corollary~\ref{cor:maj} with $\pi$ being the identity.
\end{proof}

\subsection{The general $q$-beta integral of Andrews and
Askey}

Andrews and Askey \cite{AndrewsAskey1981} generalized the $q$-beta
integral as follows: for $|q|<1$,
\begin{equation}
  \label{eq:andrews-askey}
\int_{a}^b
\frac{(qx/a;q)_\infty (qx/b;q)_\infty}
{(Ax/a;q)_\infty (Bx/b;q)_\infty} d_q x
=\frac{(1-q)(q;q)_\infty(AB;q)_\infty ab(a/b;q)_\infty(b/a;q)_\infty}
{(A;q)_\infty (B;q)_\infty (a-b) (Ba/b;q)_\infty (Ab/a;q)_\infty}.
\end{equation}

In this subsection, by computing the $q$-volume of a truncated simplex
in two different ways, we will show the following proposition which is
equivalent to the special case of \eqref{eq:andrews-askey} with
substitution
\[
(a,b,A,B)\mapsto(aq^{r-k_1},bq^{k_2+1}, q^{r+1}, q^{s+1}).
\]
\begin{prop}\label{prop:AA}
  Let $n,r,s,k_1,k_2$ be nonnegative integers such that $n=r+s+1$,
  $k_1\le r$, $k_2\le s$, and $k=k_1+k_2+1$ if $s\ge 1$ and $k=k_1$
  if $s=0$. Then 
\[
\int_a^b x^r (aq^{-k_1}/x;q)_{r} (xq^{-k_2}/b;q)_{s} d_qx\\
=\frac{[r]_q![s]_q!}{[n]_q!} b^{r+1} q^{(k-k_1)(r+1)} (aq^{-k}/b;q)_{n}.
\]
\end{prop}
\begin{proof}
Let $\pi\in S_n$.  By Corollary~\ref{cor:ab} the $q$-volume of the
truncated simplex $\order_{[a,b]^n}(P_\pi)$ is
\[
V_q(\order_{[a,b]^n}(P_\pi))=\int_{a\le x_{\pi_1} \le \dots\le x_{\pi_n}\le b} \dqx =
\frac{b^n q^{\maj(\pi)}}{[n]_q!} (aq^{-\des(\pi)}/b;q)_{n}.
\]  
Now we compute this $q$-volume in a different way by decomposing the
chain $x_{\pi_1} \le \dots\le x_{\pi_n}$ into two chains.

First we decompose $\pi$ into $\pi=\sigma n \tau$ using the largest
integer $n$. Suppose that $\sigma$ and $\tau$ have $r$ and $s$ letters
respectively and $\des(\sigma)=k_1$, $\des(\tau)=k_2$. Then $n=r+s+1$
and $k=k_1+k_2+1$ if $k_2\ge1$ and $k=k_1$ if $k_2=0$. The $q$-volume
of $\order_{[a,b]^n}(P_\pi)$ can be written as
\[
\int_a^b\left(
\int_{a\le x_{\sigma_1} \le \dots\le x_{\sigma_{r}}\le
  x_{n}} d_q y_1 \cdots d_q y_{r}
\int_{x_{n}\le x_{\tau_{1}} \le \dots\le
  x_{\tau_{s}} \le b} d_q z_1 \cdots d_q z_{s} 
\right) d_q x_n,
\]
where $y_1,\dots,y_{r}$ and $z_1,\dots,z_{s}$ are obtained by
rearranging $x_{\sigma_1},\dots,x_{\sigma_{r}}$ and
$x_{\tau_{1}},\dots,x_{\tau_{s}}$ respectively so that subscripts are
increasing. By applying Corollary~\ref{cor:ab} to the two inside
integrals, the above is equal to
\begin{equation}
  \label{eq:6}
\int_a^b
\frac{x^{r}q^{\maj(\sigma)}}{[r]_q!} (aq^{-k_1}/x;q)_{r}
\frac{b^{s}q^{\maj(\tau)}}{[s]_q!} (xq^{-k_2}/b;q)_{s}d_q x.
\end{equation}
Note that $\maj(\pi)=\maj(\sigma)+\maj(\tau)+(r+1)(k_2+1)$ if $s\ge1$
and $\maj(\pi)=\maj(\sigma)$ if $s=0$. In either case we can write
$\maj(\pi)=\maj(\sigma)+\maj(\tau)+(r+1)(k-k_1)$. This completes the
proof. 
\end{proof}

We now consider the case $s\ge1$ in Proposition~\ref{prop:AA} so that
$k=k_1+k_2+1$. One can rewrite the integral in
Proposition~\ref{prop:AA} as
\begin{multline*}
(-1)^r a^r q^{\binom r2 -k_1 r} \int_a^b 
\frac{(xq^{1-r+k_1}/a;q)_\infty (xq^{-k_2}/b;q)_\infty}
{(xq^{k_1+1}/a;q)_\infty (xq^{s-k_2}/b;q)_\infty}
d_qx\\
=(-1)^r a^r q^{\binom r2 -k_1 r}\int_{aq^{r-k_1}}^{bq^{k_2+1}}
\frac{(xq^{1-r+k_1}/a;q)_\infty (xq^{-k_2}/b;q)_\infty}
{(xq^{k_1+1}/a;q)_\infty (xq^{s-k_2}/b;q)_\infty}
d_qx,
\end{multline*}
where the equality follows from the fact that the integrand is 0 if
$x=bq^{j}$ for $0\le j\le k_2$ and $x=aq^j$ for $0\le j\le r-k_1-1$.
Thus Proposition~\ref{prop:AA} is equivalent to
\[
\int_{aq^{r-k_1}}^{bq^{k_2+1}}
\frac{(xq^{1-r+k_1}/a;q)_\infty (xq^{-k_2}/b;q)_\infty}
{(xq^{k_1+1}/a;q)_\infty (xq^{s-k_2}/b;q)_\infty}
d_qx=\frac{(-1)^r b^{r+1} [r]_q![s]_q! 
q^{kr+k-k_1-\binom r2} (aq^{-k}/b;q)_{n}}{a^r[n]_q!}.
\]

This is the $(a,b,A,B)\mapsto(aq^{r-k_1},bq^{k_2+1}, q^{r+1}, q^{s+1})$ case of 
\eqref{eq:andrews-askey}. 

\subsection{$q$-integrals of monomials over the order polytope of a
  forest poset}

In this subsection we consider special posets called forests and evaluate
the $q$-integral of a monomial over the order polytope coming from
these posets. These forests have an order polytope whose $q$-volume 
has a hook formula, see Corollary~\ref{cor:hookforest}.

\begin{defn}
  A poset is called a \emph{forest} if every element is covered by at
  most one element. For a forest poset $F$ and its element $x$, the
  \emph{hook length} $h_F(x)$ of $x$ is defined to be the number of
  elements $y\in F$ such that $y\le_F x$. An element of a forest is
  called \emph{isolated} if it has no relation with other elements. A
  \emph{leaf} is a non-isolated element whose hook length is $1$.
\end{defn}

\begin{exam}
  Let $F$ be the forest in Figure~\ref{fig:forest1}. Then $x_6$
  is an isolated element and $x_1,x_2,x_4,x_7,x_8$ are leaves.  These
  elements have hook length $1$. The hook lengths of other elements
  are $h_F(x_3)=3$, $h_F(x_5)=5$, $h_F(x_9)=3$.
\end{exam}

\begin{figure}
  \centering
\includegraphics{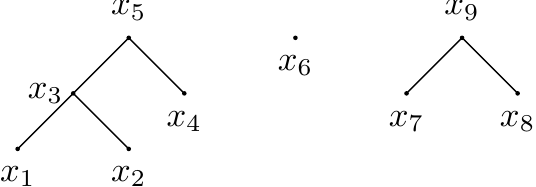}  
  \caption{A forest poset.}
  \label{fig:forest1}
\end{figure}

\begin{defn}
For a forest $F$ and a sequence $a=(a_1,\dots,a_n)$ of
nonnegative integers, let $F_a$ be the poset obtained from $F$ by
attaching $a_i$ leaves to $x_i$ for each $i\in[n]$, see
Figure~\ref{fig:forest2}.
\end{defn}

\begin{exam}
  Let $F$ be the poset in Figure~\ref{fig:forest1} and
\[
a=(a_1,\dots,a_9)=(0,3,2,1,2,3,0,2,1).
\]
The forest $F_a$ is shown in Figure~\ref{fig:forest2}, where the
short edges are newly added from $F$ in Figure~\ref{fig:forest1}. 
We have
\[
h_{F_a}(x_1)=1, \quad
h_{F_a}(x_2)=4, \quad
h_{F_a}(x_3)=8, \quad
h_{F_a}(x_4)=2, \quad
h_{F_a}(x_5)=13, 
\]
\[
h_{F_a}(x_6)=4, \quad
h_{F_a}(x_7)=1, \quad
h_{F_a}(x_8)=3, \quad
h_{F_a}(x_9)=6.
\]
\end{exam}
 
\begin{figure}
  \centering
\includegraphics{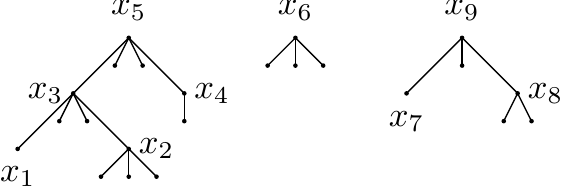}  
  \caption{The forest poset $F_a$, where $F$ is the forest in
    Figure~\ref{fig:forest1} and $a=(a_1,\dots,a_9)$.}
  \label{fig:forest2}
\end{figure}

The next corollary allows us to evaluate any multiple $q$-integral of
a monomial over the order polytope of a naturally labeled forest. 

\begin{thm}\label{thm:forest}
  Let $F$ be a forest on $\{x_1,\dots,x_n\}$ with labeling $\omega_n$
  given by $\omega_n(x_i)=i$. Suppose that $\omega_n$ is a natural
  labeling.  Let $a=(a_1,\dots,a_n)$ be a sequence of nonnegative
  integers.  Then, we have
\[
\int_{\order(F)} x_1^{a_1}\cdots x_n^{a_n} \dqx
=\prod_{v\in F_a} \frac{1}{[h_{F_a}(v)]_q}.
\]
\end{thm}
\begin{proof}
  We prove this using induction on $n$. If $n=1$, then both sides are
  equal to $1/[a+1]_q$. Suppose that $n>1$ and the theorem is true for
  $n-1$. Since $\omega$ is a natural labeling, $x_1$ is an isolated
  element or a leaf in $F$.

  Case 1: $x_1$ is an isolated element in $F$. Then the range for
  $x_1$ in the $q$-integral is $0\le x_1\le 1$. Let $F'=F-\{x_1\}$ and
  $a'=(a_2,\dots,a_n)$.  By the induction hypothesis, we have
\begin{multline*}
\int_{\order(F)} x_1^{a_1}\cdots x_n^{a_n} \dqx
=\int_{0}^1 x_1^{a_1}d_qx_1 \int_{\order(F')} x_2^{a_2}\cdots x_n^{a_n} 
d_qx_2\cdots d_qx_n\\
=\frac{1}{[a_1+1]_q}\prod_{v\in F'_{a'}} \frac{1}{[h_{F'_{a'}}(v)]_q}
=\prod_{v\in F_a} \frac{1}{[h_{F_a}(v)]_q}.
\end{multline*}
  
Case 2: $x_1$ is a leaf in $F$. Let $x_k$ be the unique element
covering $x_1$. Then the range for $x_1$ in the $q$-integral is
$0\le x_1\le x_k$. Let $F'=F-\{x_1\}$ and $a'=(a'_2,\dots,a'_n)$,
where $a'_i=a_i$ if $i\ne k$ and $a'_k=a_k+a_1+1$.
By the induction hypothesis, we have
\begin{multline*}
\int_{\order(F)} x_1^{a_1}\cdots x_n^{a_n} \dqx
= \int_{\order(F')} \left(\int_{0}^{x_k} x_1^{a_1}d_qx_1\right)
x_2^{a_2}\cdots x_n^{a_n} 
d_qx_2\cdots d_qx_n\\
= \int_{\order(F')} \left(\frac{x_k^{a_1+1}}{[a_1+1]_q}\right)
x_2^{a_2}\cdots x_n^{a_n} 
d_qx_2\cdots d_qx_n\\
= \frac{1}{[a_1+1]_q}\int_{\order(F')} 
x_2^{a'_2}\cdots x_n^{a'_n} 
d_qx_2\cdots d_qx_n
=\frac{1}{[a_1+1]_q}\prod_{v\in F'_{a'}} \frac{1}{[h_{F'_{a'}}(v)]_q}.
\end{multline*}
It is easy to check that
\[
\frac{1}{[a_1+1]_q}\prod_{v\in F'_{a'}} \frac{1}{[h_{F'_{a'}}(v)]_q}
=\prod_{v\in F_a} \frac{1}{[h_{F_a}(v)]_q},
\]
which completes the proof.   
\end{proof}

If $a=(0,\dots,0)$ in Theorem~\ref{thm:forest} we obtain the
$q$-volume of the order polytope of a forest.  Using
Theorem~\ref{thm:pw} we also obtain the maj-generating function for
the linear extensions of a forest, which was first proved by
Bj\"orner and Wachs \cite{Bjorner1989}.

\begin{cor}
\label{cor:hookforest}
  Let $F$ be a forest on $\{x_1,\dots,x_n\}$ with labeling
  $\omega_n$ given by $\omega_n(x_i)=i$. Suppose that $\omega_n$ is
  naturally labeled. Then, we have
\[
V_q(\order(F)) = \frac{1}{\prod_{v\in F} [h_F(v)]_q},
\]
\[
\sum_{\pi\in\LL(F,\omega_n)} q^{\maj(\pi)}
= \frac{[n]_q!}{\prod_{v\in F} [h_F(v)]_q}.
\]
\end{cor}

\section{$q$-Selberg integrals}
\label{sec:q-selberg-integrals}

In this section we will find a combinatorial interpretation for a
$q$-Selberg integral. We will use the first three lemmas of
Section~\ref{sec:q-integrals-over} to ``insert" linear factors in
$q$-integrals by building a ``Selberg poset'' whose $q$-volume is the
$q$-Selberg integral, up to a constant. As a corollary of the
$q$-Selberg integral evaluation, we obtain the factorization of the
$maj$-generating function for the linear extensions, see
Corollary~\ref{cor:majSelberg}.

There are many generalizations of the Selberg integral, see
\cite{Forrester2008}. We consider the following $q$-Selberg integral.
Recall \cite[p. 493]{AAR_SP} that the $q$-gamma function is defined by
\[
\Gamma_q(x) = \frac{(q;q)_{x-1}}{(1-q)^{x-1}}.
\]

\begin{defn}[$q$-Selberg integral]
  For complex numbers $\alpha$ and $\beta$ with
  $\mathrm{Re} (\alpha)>0$, $\mathrm{Re} (\beta)>0$, and a
  non-negative integer $m$, the \emph{$q$-Selberg integral}
  $S_q(n,\alpha,\beta,m)$ is defined by
\[
\int_{0\le x_1\le\cdots\le x_n\le1}
 \prod_{i=1}^n x_i^{\alpha-1} (qx_i;q)_{\beta-1} 
\prod_{1\le i< j\le n} 
x_j^{2m-1} \left( q^{1-m} x_i/x_j ; q \right)_{2m-1}
\DeltaBar(x) \dqx.
\]
\end{defn}

The following theorem was conjectured by Askey \cite{Askey1980} and
proved independently by Habsieger \cite{Habsieger1988} and Kadell
\cite{Kadell1988b}.

\begin{thm}
\label{thm:AHKS}
Let $\alpha$ $\beta$ be complex numbers with $\mathrm{Re} (\alpha)>0$,
$\mathrm{Re} (\beta)>0$, and $m$ a non-negative integer. Then 
\begin{equation}\label{eq:q-Selberg}
S_q(n,\alpha,\beta,m) =  q^{\alpha m \binom{n}{2} + 2 m^2 \binom{n}{3}}
\prod_{j=1}^n \frac{\Gamma_q(\alpha+(j-1)m)
\Gamma_q(\beta+(j-1)m)\Gamma_q(jm)}
{\Gamma_q(\alpha+\beta+(n+j-2)m)\Gamma_q(m)}.
\end{equation}
\end{thm}

We note that \eqref{eq:q-Selberg} is a slight
modification of Askey's original conjecture:
\begin{multline}
\label{eq:Askey}
\int_0^1 \cdots \int_0^1
 \prod_{i=1}^n  x_i^{\alpha-1}  (q x_i ; q)_{\beta-1}
 \prod_{1 \le i < j \le n}
  x_j^{2m} \left( q^{1-m} x_i/x_j ; q \right)_{2m}
 d_q x_1 \cdots d_q x_n\\
= q^{\alpha m \binom{n}{2} + 2 m^2 \binom{n}{3}}
\prod_{j=1}^n \frac{\Gamma_q(\alpha+(j-1)m)
\Gamma_q(\beta+(j-1)m)\Gamma_q(1+jm)}
{\Gamma_q(\alpha+\beta+(n+j-2)m)\Gamma_q(1+m)}.
\end{multline}

The equivalence of \eqref{eq:q-Selberg} and \eqref{eq:Askey} can be
obtained from the fact that the integrand in $S_q(n,\alpha,\beta,m)$
is symmetric and Kadell's result \cite{Kadell1988a}: for an
anti-symmetric function $f(x)$, we have
\[
\int_{[a,b]^n} f(x)\prod_{1\le i<j\le n} (x_i- p x_j) d_qx
=\frac{[n]_p!}{n!}\int_{[a,b]^n} f(x) \DeltaBar(x) d_qx.
\]

We will give a combinatorial interpretation of
$S_q(n,\alpha,\beta,m)$ when $\alpha-1=r$, $\beta-1=s$ and $m$ are
non-negative integers. This requires defining a poset, the Selberg poset, whose 
order polytope has $q$-volume given by the $q$-Selberg integral. 
Basically we start with a chain $C$ of $n$ elements. Insert two independent chains 
with $m$ elements between any two elements of $C$. For each element of $c\in C$, 
insert a chain with $s$ elements above $c$ 
and a chain with $r$ elements below $c$.  

\begin{defn}
\label{defn:selbergposet}
We define the \emph{Selberg poset} $P(n,r,s,m)$ to be the poset in
which the elements are $x_i, y_i^{(a)}, z_i^{(b)},w_{i,j}^{(k)}$ for
$i,j\in[n], a\in[r], b\in[s], k\in[m]$ with $i\ne j$, and the covering
relations are as follows:
\begin{itemize}
\item $x_i< w_{i,j}^{(1)}< \cdots < w_{i,j}^{(m)} < x_j$ for
  $1\le i<j\le n$,
\item $x_i< w_{j,i}^{(m)}< \cdots < w_{j,i}^{(1)} < x_j$ for
$1\le i<j\le n$,
\item $y_{i}^{(1)} < \cdots< y_{i}^{(r)} < x_i 
< z_i^{(1)}<\cdots< z_i^{(s)}$ for $1\le i\le n$.
\end{itemize}
We define $W$ to be the following permutation of the elements of
$P(n,r,s,m)$:
\begin{multline}\label{eq:W}
W= \left(\prod_{1\le i< j\le n} w_{i,j}^{(1)} \cdots w_{i,j}^{(m)}\right)
\left(\prod_{1\le i<j\le n} w_{j,i}^{(1)} \cdots w_{j,i}^{(m)}\right)\\
\left(\prod_{i=1}^n  y_i^{(1)}\cdots  y_i^{(r)} \right)
\left(\prod_{i=1}^n x_i \right)
\left(\prod_{i=1}^n  z_i^{(1)}\cdots  z_i^{(s)}\right),
\end{multline}
where $\prod_{1\le i< j\le n} w_{i,j}^{(1)} \cdots w_{i,j}^{(m)}$
means the concatenation of the word
$w_{i,j}=w_{i,j}^{(1)} \cdots w_{i,j}^{(m)}$ for $1\le i<j\le n$ so
that $w_{i,j}$ appears before $w_{i',j'}$ if $i<i'$ or ($i=i'$ and
$j<j'$). The other products are defined in the same way, for instance,
$\prod_{i=1}^n x_i$ means the word $x_1\cdots x_n$.  Finally, we
define $\omega$ to be the labeling of $P(n,r,s,m)$ such that for
$u\in P(n,r,s,m)$, $\omega(u)$ is the position of $u$ in $W$.
\end{defn}

\begin{exam}
Let $P$ be the Selberg poset $P(n,r,s,m)$ in
Figure~\ref{fig:selberg_poset}. 
Then 
\begin{multline*}
W=  
w_{1,2}^{(1)} w_{1,2}^{(2)} w_{1,2}^{(3)}
w_{1,3}^{(1)} w_{1,3}^{(2)} w_{1,3}^{(3)}
w_{1,4}^{(1)} w_{1,4}^{(2)} w_{1,4}^{(3)}
w_{2,3}^{(1)} w_{2,3}^{(2)} w_{2,3}^{(3)}
w_{2,4}^{(1)} w_{2,4}^{(2)} w_{2,4}^{(3)}
w_{3,4}^{(1)} w_{3,4}^{(2)} w_{3,4}^{(3)}\\
w_{2,1}^{(1)} w_{2,1}^{(2)} w_{2,1}^{(3)}
w_{3,1}^{(1)} w_{3,1}^{(2)} w_{3,1}^{(3)}
w_{4,1}^{(1)} w_{4,1}^{(2)} w_{4,1}^{(3)}
w_{3,2}^{(1)} w_{3,2}^{(2)} w_{3,2}^{(3)}
w_{4,2}^{(1)} w_{4,2}^{(2)} w_{4,2}^{(3)}
w_{4,3}^{(1)} w_{4,3}^{(2)} w_{4,3}^{(3)}\\
y_1^{(1)} y_1^{(2)}
y_2^{(1)} y_2^{(2)}
y_3^{(1)} y_3^{(2)}
y_4^{(1)} y_4^{(2)}
x_1 x_2 x_3 x_4
z_1^{(1)}
z_2^{(1)}
z_3^{(1)}
z_4^{(1)}.
\end{multline*}
The labeling $\omega$ of $P(n,r,s,m)$ is shown in
Figure~\ref{fig:selberg_poset}. 
\end{exam}

\begin{figure}
  \centering
\includegraphics[width=\textwidth]{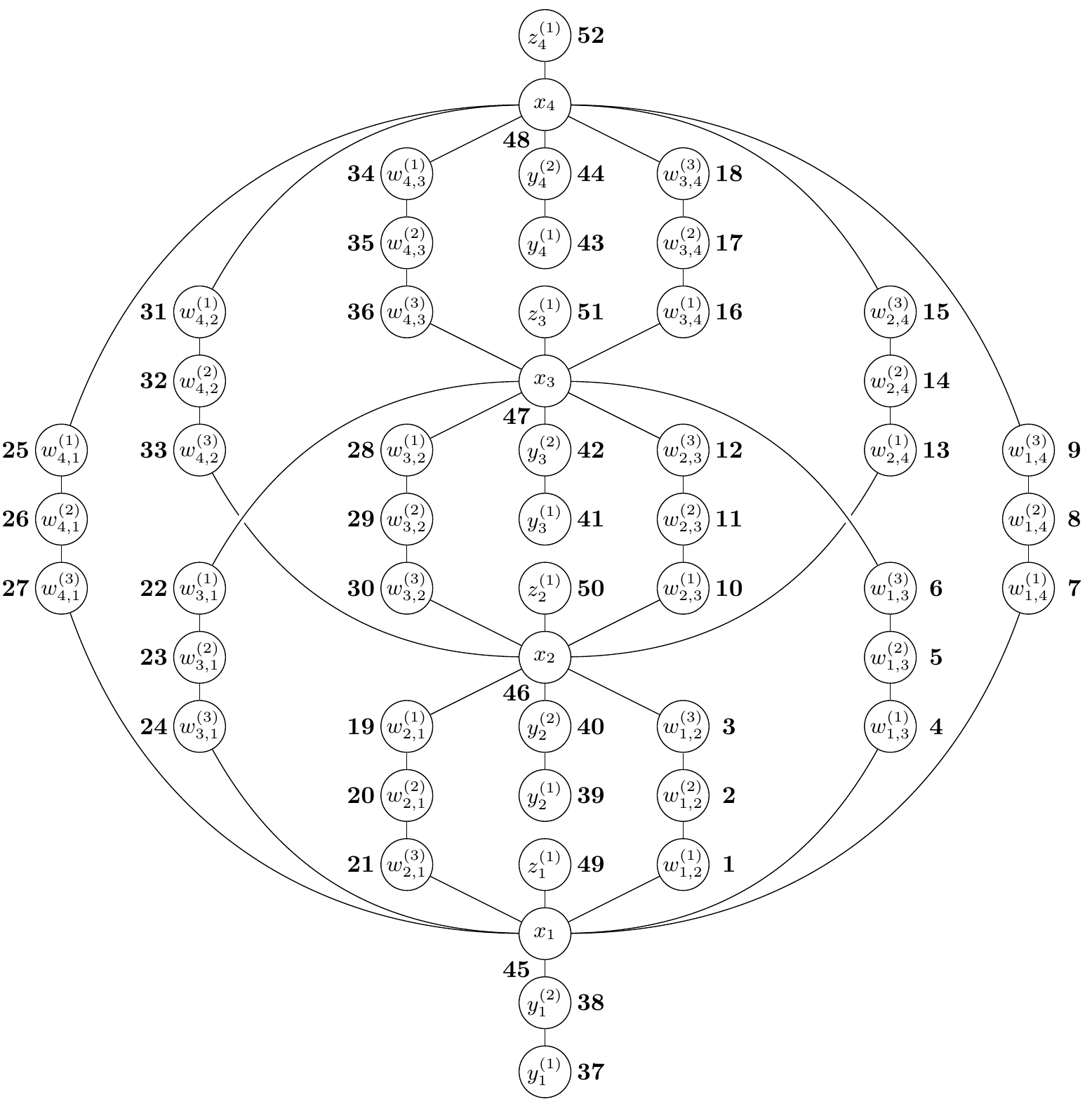}    
\caption{The Selberg poset $P(n,r,s,m)$ for $n=4, r=2, s=1, m=3$ with
  labeling.}
  \label{fig:selberg_poset}
\end{figure}

The following theorem implies that the $q$-Selberg integral
is a $q$-volume of the order polytope $\order(P(n,r,s,m))$ up to a certain
factor.

\begin{thm}
\label{thm:selbergvol}
We have
  \begin{multline*}
\int_{0\le x_1\le\cdots\le x_n\le1}
 \prod_{i=1}^n x_i^{r} (qx_i;q)_{s}
\prod_{1\le i< j\le n} 
x_j^{2m-1} \left( q^{1-m} x_i/x_j ; q \right)_{2m-1}
\DeltaBar(x) \dqx\\
=q^{-\binom m2\binom n2} 
([r]_q!)^n ([s]_q!)^n ([m]_q!)^{2\binom n2}
\cdot \int_{\order(P(n,r,s,m))} d_qW,
  \end{multline*}
  where $d_qW$ is the order of integration obtained by putting $d_q$
  in front of each letter in \eqref{eq:W}.
\end{thm}

\begin{proof}
  Observe that the left side is equal to
\[
\int_{0\le x_1\le\cdots\le x_n\le1}
 \prod_{i=1}^n x_i^{r} (qx_i;q)_{s}
\prod_{1\le i< j\le n} 
x_j^{m} \left( q^{1-m} x_i/x_j ; q \right)_{m}
x_j^{m} \left(x_i/x_j ; q \right)_{m} \dqx.
\]
Then we get the right side by applying
the happy cat lemma to each factor $(qx_i;q)_s$,
the scaredy cat lemma to each factor $x_i^r$,
the first part \eqref{eq:xyx} of attaching chain lemma
with $\rho=m(m-1)\cdots 1$ to each factor
$x_j^{m} \left( q^{1-m} x_i/x_j ; q \right)_{m}$
 and the first part \eqref{eq:xyx} of attaching chain lemma
with $\rho=12\cdots m$ to each factor
$x_j^{m} \left(x_i/x_j ; q \right)_{m}$.
\end{proof}

The following corollary gives a combinatorial interpretation for the
$q$-Selberg integral in terms of linear extensions.

\begin{cor}
\label{cor:comb_inter}
We have
  \begin{multline*}
\int_{0\le x_1\le\cdots\le x_n\le1}
 \prod_{i=1}^n x_i^{r} (qx_i;q)_{s} 
\prod_{1\le i< j\le n} 
x_j^{2m-1} \left( q^{1-m} x_i/x_j ; q \right)_{2m-1}
\DeltaBar(x) \dqx\\
=\frac{q^{-\binom m2\binom n2} 
([r]_q!)^n ([s]_q!)^n ([m]_q!)^{2\binom n2}}
{[N]_q!} \sum_{\pi\in\LL(P(n,r,s,m),\omega)}q^{\maj(\pi)},
  \end{multline*}
where
\[
N = n(r+s+1) + 2m \binom n2,
\]
and, $\omega$ is the labeling of the Selberg poset $P(n,r,s,m)$
defined in Definition~\ref{defn:selbergposet}.
\end{cor}
\begin{proof}
This is an immediate consequence of Theorem~\ref{thm:selbergvol}
and Theorem~\ref{thm:pw}.
\end{proof}

Using the $q$-Selberg integral formula, we obtain an explicit product
formula for the maj-generating function for linear extensions of the
Selberg poset. 

\begin{cor}
\label{cor:majSelberg}
We have
\begin{multline*}
\sum_{\pi\in\LL(P(n,r,s,m),\omega)}q^{\maj(\pi)} \\
=\frac{q^{\binom m2\binom n2+(r+1)m \binom{n}{2} + 2 m^2 \binom{n}{3}} [N]_q!}
{([r]_q!)^n ([s]_q!)^n ([m]_q!)^{2\binom n2}}
\prod_{j=1}^n \frac{[r+(j-1)m]_q![s+(j-1)m]_q! [jm-1]_q!}
{[r+s+1+(n+j-2)m]_q![m-1]_q!},
\end{multline*}
where $N$ and $\omega$ are the same as in
Corollary~\ref{cor:comb_inter}.
\end{cor}
\begin{proof}
This is an immediate consequence of Corollary~\ref{cor:comb_inter}
and Theorem~\ref{thm:AHKS}.
\end{proof}

We note that Kim and Oh \cite{KimOh} used Stanley's combinatorial
interpretation to study the Selberg integral. They found a connection
with certain combinatorial objects called Young books, which
generalizes standard Young tableaux of certain shapes. It would be
interesting to see if their results can be generalized using ours. We
also note that Kim and Okada \cite{KimOkada} considered a different
$q$-Selberg integral. By evaluating the $q$-integral, they showed that
it is a generating function for Young books with certain weight. 

\section{Reverse plane partitions}
\label{sec:reverse-plane-part}

In this section we consider $q$-Selberg type integrals with Schur
functions in the integrand. We then relate the resulting integral evaluations to 
generating functions for reverse plane partitions and generalized Gelfand-Tsetlin
patterns. In particular we
\begin{enumerate}
\item define a Schur poset whose truncated order polytope has $q$-volume 
which is basically a Schur function $s_\lambda(x_1,\dots,x_n)$ 
(Lemma~\ref{lem:schurDelta}), 
\item use $q$-integral evaluations to give generating functions for shifted and 
non-shifted reverse plane partitions with a given shape and diagonal 
(Theorem~\ref{thm:rpp} and Corollary~\ref{cor:rpp2}), 
\item show that a $q$-integral of Warnaar may be interpreted as the trace 
generating function for reverse plane partitions of given shape 
(Theorems~\ref{thm:warnaar} and \ref{thm:warnaar_rpp}), and 
give an analogous Warnaar-type $q$-integral for the shifted version due to Gansner 
(Theorem~\ref{thm:gansner_int}),
\item define a generalized Gelfand-Tsetlin pattern, and give a 
new generating function for these objects (Theorem~\ref{thm:GT}),
\item relate the Askey-Kadell-Selberg integral to a new weighted generating function
for reverse plane partitions of square shape 
(Theorems~\ref{thm:AKS_RPP} and \ref{thm:selbergRPP}).
\end{enumerate}  

Before we get to the specific subsections,  we need some notation and 
definitions.

For a partition $\lambda$ with $\ell(\lambda)\le n$, let
$(\delta_{n+1}+\lambda)^*$ be the shifted Young diagram obtained from
$\lambda$ by attaching the shifted staircase $\delta_{n+1}=(n,n-1,\dots,1)$.

\begin{defn}
  Let $\lambda$ be a Young diagram or shifted Young diagram.  A
  \emph{reverse plane partition} of shape $\lambda$ is a filling of
  $\lambda$ with non-negative integers which are weakly increasing along rows
  and columns. We denote by $\RPP(\lambda)$ the set of reverse plane
  partitions of shape $\lambda$, see Figure~\ref{fig:2RPPs}. 
\end{defn}

\begin{defn}
  Let $\lambda$ be a partition or a shifted partition. The
  \emph{$k$-diagonal} of $\lambda$ is the set of cells in $\lambda$
  which are in row $i$ and column $j$ with $j-i=k$.  For
  $T\in\RPP(\lambda)$, the \emph{reverse diagonal} of $T$ is the
  sequence $\rdiag(T)=(\mu_1,\dots,\mu_n)$ obtained by reading the
  entries of $T$ in the $0$-diagonal of $\lambda$ from southeast to
  northwest. Also the \emph{trace} is defined by $\tr(T)=\mu_1+\dots+\mu_n$,
  and the sum of entries of $T$ is denoted $|T|.$
\end{defn}

\begin{figure}
  \centering
\includegraphics{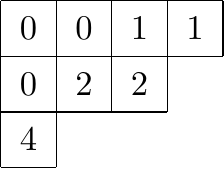} \qquad
\includegraphics{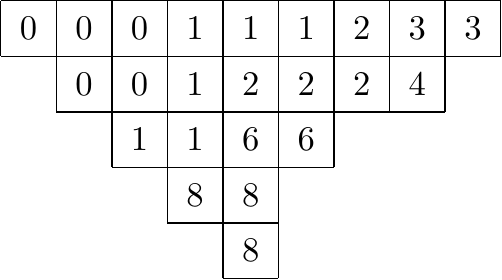} 
  \caption{The left diagrams is an element in $\RPP((4,3,1))$ and the
    right diagram is an element in $\RPP((\delta_{5+1}+(4,3,1))^*)$.}
\label{fig:2RPPs}
\end{figure}

\begin{exam}
  Let $T_1$ be the left diagram and $T_2$ the right diagram in
  Figure~\ref{fig:2RPPs}. Then $\tr(T_1)=2$, $\tr(T_2)=17,$ 
  $|T_1|=10,$ $|T_2|=60,$ 
  $\rdiag(T_1)=(2,0)$ and
  $\rdiag(T_2)=(8,8,1,0,0)$. Note that the reverse diagonal of a
  reverse plane partition is always a partition.
\end{exam}

\begin{defn}\label{defn:schur_poset}
  Let $\lambda$ be a partition with $\ell(\lambda)\le n$.  For
  $0\le k\le \lambda_1+n-1$, let $a_k$ be the size of the $k$-diagonal
  of $(\delta_{n+1}+\lambda)^*$. We label the cells in the
  $k$-diagonal $(\delta_{n+1}+\lambda)^*$ by
  $x^{(k)}_1,\dots,x^{(k)}_{a_k}$ from southeast to northwest.  The
  \emph{0-Schur poset} $P^0_{\schur}(n,\lambda)$ is the poset on
\[
\{x^{(k)}_i: 0\le k\le \lambda_1+n-1,\quad 1\le i\le a_k\}
\]
with relations $x^{(k_1)}_{i_1}\le x^{(k_2)}_{i_2}$ if the cell
$x^{(k_2)}_{i_2}$ is weakly northwest of the cell $x^{(k_1)}_{i_1}$ in
$(\delta_{n+1}+\lambda)^*$.  See Figure~\ref{fig:schur}.  Then we
define the \emph{Schur poset} $P_{\schur}(n,\lambda)$ to be the poset
obtained from the \emph{0-Schur poset} $P^0_{\schur}(n,\lambda)$
by removing the elements $x_1^{(0)}, \dots, x_n^{(0)}$.
\end{defn}

\begin{figure}
  \centering
\includegraphics{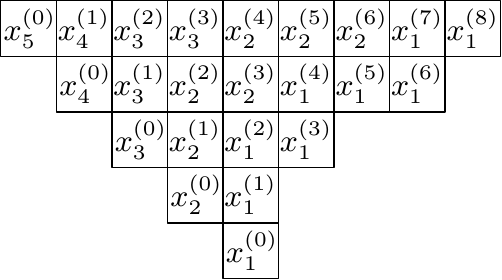}\qquad \qquad
\includegraphics{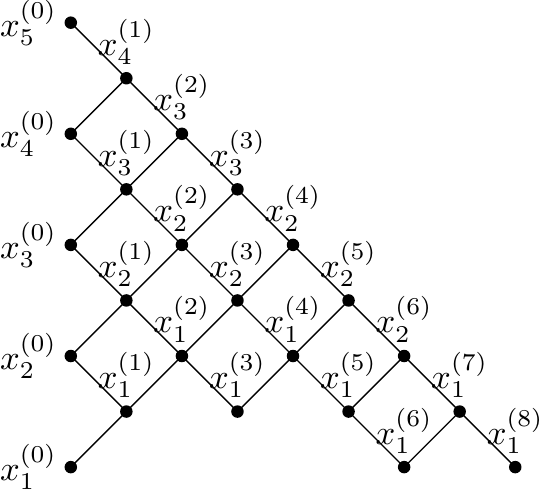}
\caption{The left diagram shows the labels of the cells of
  $(\delta_{n+1}+\lambda)^*$, and the right diagram shows the Hasse
  diagram of the 0-Schur poset $P^0_{\schur}(n,\lambda)$ for $n=5$ and
  $\lambda=(4,3,1)$.}
\label{fig:schur}
\end{figure}
 
Note that we have
\begin{align*}
P_{\schur}(n,\lambda) &\cong P^0_{\schur}(n-1,\lambda)
\qquad \mbox{if $\ell(\lambda)\le n-1$,}\\
P_{\schur}(n,\lambda) &\cong P^0_{\schur}(n,\lambda-(1^n))
\qquad \mbox{if $\ell(\lambda)= n$.}
\end{align*}

The next proposition shows that the Schur function $s_\lambda(x)$ can
be written as a $q$-volume of a truncated order polytope of a Schur
poset up to a constant factor.

\begin{prop}\label{prop:schur_poset}
  Let $\lambda$ be a partition with $\ell(n)\le n$ and $a_k$ the size
  of the $k$-diagonal of $(\delta_{n+1}+\lambda)^*$.  Let
  $x=(x_1,\dots,x_n)$ be a sequence of fixed real numbers
  $x_1<x_2<\dots<x_n$ and 
\[
I =\left\{
  \begin{array}{ll}
 \{(x^{(k)}_i)\in \mathbb R^{\binom n2 +|\lambda|}: 
x_j \le x_j^{(1)} \le x_{j+1}, 1\le j\le n-1\},
    & \mbox{if $\ell(\lambda)\le n-1$,}\\
 \{(x^{(k)}_i)\in \mathbb R^{\binom n2 +|\lambda|}: 
x_{j-1} \le x_j^{(1)} \le x_{j}, 1\le j\le n, x_0=0\},
    & \mbox{if $\ell(\lambda)=n$,}
  \end{array} \right.
\]
where $(x^{(k)}_i)$ means a point in
$\mathbb R^{\binom n2 +|\lambda|}$ with indices
$1\le k\le \lambda_1+n-1$ and $1\le i\le a_k$.  Then 
\[
s_\lambda(\vec x)\DeltaBar(\vec x) = \prod_{j=1}^{n}[\lambda_j+n-j]_q!
\int_{\order_I(P_{\schur}(n,\lambda))} 
d_q\vec x^{(\lambda_1+n-1)} d_q\vec x^{(\lambda_1+n-2)} \cdots
d_q\vec x^{(1)},
\]
where $d_qx^{(k)} = d_qx_1^{(k)}\cdots d_qx_{a_k}^{(k)}$. 
\end{prop}

We will prove Lemma~\ref{lem:schurDelta} which is equivalent to
Proposition~\ref{prop:schur_poset}.  The idea is to introduce new sets
of variables which are interlacing with the previous sets of
variables.

\begin{lem}
\label{lem:schurDelta}
  Let $\lambda$ be a partition with $\ell(n)\le n$ and $a_k$ the size
  of the $k$-diagonal of $(\delta_{n+1}+\lambda)^*$.  For
  $0\le k\le \lambda_1+n-1$, we consider sequences of variables
  $\vec x^{(k)} = (x^{(k)}_1,\dots,x^{(k)}_{a_k})$ with $x^{(0)}_i=x_i$.
Then for the set
\[
Q =\{\vec x^{(\lambda_1+n-1)}\prec \cdots \prec \vec x^{(0)}\}
\]
of inequalities, we have
\begin{equation}
  \label{eq:33}
s_\lambda(\vec x)\DeltaBar(\vec x) = \prod_{j=1}^{n}[\lambda_j+n-j]_q!
\int_Q d_q\vec x^{(\lambda_1+n-1)} d_q\vec x^{(\lambda_1+n-2)} \cdots
d_q\vec x^{(1)},
\end{equation}
where $d_qx^{(k)} = d_qx_1^{(k)}\cdots d_qx_{a_k}^{(k)}$. 
\end{lem}

\begin{proof}
  We use induction on $n$. If $n=1$, the left hand side of
  \eqref{eq:33} is
  $s_{\lambda}(x_1) = x_1^{\lambda_1}$. By Corollary~\ref{cor:maj},
  the right hand side of \eqref{eq:33} is
\[
[\lambda_1]_q! \int_{0\le x_1^{(\lambda_1)}\le \cdots\le x_1^{(1)}\le
  x_1} d_q x_1^{(\lambda_1)} \cdots d_q x_1^{(1)}
=x_1^{\lambda_1}.
\]
Now suppose that \eqref{eq:33} is true for $n-1$ and consider $n$. 

\textsl{CASE 1}: $\ell(\lambda)\le n-1$. By Lemma~\ref{lem:schur1}, we
have
\begin{equation}
  \label{eq:2}
s_\lambda(x) \DeltaBar(x)
=\prod_{j=1}^{n-1} [\lambda_j+n-j]_q
\int_{w\prec x} s_\lambda(w) \DeltaBar(w) 
d_q w_1 \cdots d_q w_{n-1}. 
\end{equation}
Since $\ell(\lambda)\le n-1$ and $w=(w_1,\dots,w_{n-1})$ has $n-1$
variables, by the induction hypothesis, we have
\begin{equation}
  \label{eq:3}
s_\lambda(w)\DeltaBar(w) = \prod_{j=1}^{n-1}[\lambda_j+n-j-1]_q!
\int_{Q'} d_qw^{(\lambda_1+n-2)} d_qw^{(\lambda_1+n-3)} \cdots d_qw^{(1)},
\end{equation}
where $w^{(k)}=(w_1^{(k)},\dots,w_{a'_k}^{(k)})$ is a sequence of
variables for $0\le k\le \lambda_1+n-2$ with $w^{(0)}_i = w_i$, 
$a'_{k}$ is the size of the $k$-diagonal of
$(\delta_{n}+\lambda)^*$, and
\[
Q' =\{w^{(\lambda_1+n-2)}\prec \cdots \prec w^{(0)}\}.
\]
 Note that $a'_k=a_{k+1}$. 
By substituting $w_i^{(k)}=x_i^{(k+1)}$, we have
\begin{multline}\label{eq:7}
\int_{w\prec x} 
\left(  \int_{Q'} d_qw^{(\lambda_1+n-2)} d_qw^{(\lambda_1+n-3)} \cdots
d_qw^{(1)}\right) d_q w_1 \cdots d_q w_{n-1}  \\
=\int_Q d_qx^{(\lambda_1+n-1)} d_qx^{(\lambda_1+n-2)} \cdots d_qx^{(1)}.
\end{multline}
By \eqref{eq:2}, \eqref{eq:3}, and \eqref{eq:7}, we have
\[
s_\lambda(x) \DeltaBar(x) = \prod_{j=1}^{n-1}[\lambda_j+n-j]_q!
\int_Q d_qx^{(\lambda_1+n-1)} d_qx^{(\lambda_1+n-2)} \cdots d_qx^{(1)}.
\]
Since $\lambda_n=0$, we have
\[
\prod_{j=1}^{n-1}[\lambda_j+n-j]_q! =
\prod_{j=1}^{n}[\lambda_j+n-j]_q!, 
\]
which implies that \eqref{eq:33} is also true for $n$.

\textsl{CASE 2}: $\ell(\lambda) = n$. Let $\lambda_n=m$.  Note that
for $0\le k\le m$, we have $x^{(k)}=(x^{(k)}_1,\dots,x^{(k)}_n)$.  By
Lemma~\ref{lem:schur2}, for $0\le k\le m-1$, we have
\begin{equation}
  \label{eq:8}
s_{\lambda-(k^n)}(x^{(k)}) \DeltaBar(x^{(k)})
=\prod_{j=1}^{n} [\lambda_j-k+n-j]_q
\int_{x^{(k+1)}\prec x^{(k)}} s_{\lambda-((k+1)^n)}(x^{(k+1)}) \DeltaBar(x^{(k+1)}) 
d_q x^{(k+1)}. 
\end{equation}
Applying \eqref{eq:8} for $k=0,1,\dots,m-1$, iteratively, we get
\begin{equation}
  \label{eq:14}
s_{\lambda}(x) \DeltaBar(x)
=\prod_{k=0}^{m-1}\prod_{j=1}^{n} [\lambda_j-k+n-j]_q
\int_{x^{(m)}\prec \cdots\prec x^{(0)}} s_{\lambda-(m^n)}(x^{(m)}) \DeltaBar(x^{(m)}) 
d_q x^{(m)}\cdots d_q x^{(1)}.   
\end{equation}
Since $\lambda-(m^n)$ has length at most $n-1$, by 
\textsl{CASE 1}, we get
\begin{equation}
  \label{eq:31}
s_{\lambda-(m^n)}(x^{(m)}) \DeltaBar(x^{(m)}) 
= \prod_{j=1}^{n}[\lambda_j-m+n-j]_q!
\int_{Q'} d_qx^{(\lambda_1+n-1)} d_qx^{(\lambda_1+n-2)} \cdots d_qx^{(m+1)},
\end{equation}
where $Q' = \{x^{(\lambda_1+n-1)}\prec x^{(\lambda_1+n-2)}\prec \cdots \prec x^{(m)}\}$.
By \eqref{eq:14} and \eqref{eq:31}, we have that $s_{\lambda}(x)
\DeltaBar(x)$ is equal to 
\[
\left(\prod_{k=0}^{m-1}\prod_{j=1}^{n}
[\lambda_j-k+n-j]_q \right) \prod_{j=1}^{n}[\lambda_j-m+n-j]_q!
\int_Q d_qx^{(\lambda_1+n-1)} d_qx^{(\lambda_1+n-2)} \cdots d_qx^{(1)},
\]
which is equal to the right hand side of \eqref{eq:33}.  Thus
\eqref{eq:33} is also true for $n$. By induction \eqref{eq:33} is true
for all nonnegative integers $n$. 
\end{proof}

\subsection{Shifted reverse plane partitions with fixed diagonal
  entries}

The main result in this subsection, Theorem~\ref{thm:rpp}, reinterprets
the $q$-integral evaluation in Lemma~\ref{lem:schurDelta} as a
generating function for shifted reverse plane partitions.

\begin{thm}
\label{thm:rpp}
For $\lambda,\mu\in\Par_n$ we have
\[
\sum_{\substack{T\in \RPP ( (\delta_{n+1}+\lambda)^* ) \\
\rdiag(T)=\mu}}  q^{|T|}
=  \frac{q^{-b(\delta_{n+1}+\lambda)}}
{\prod_{j=1}^{n}(q;q)_{\lambda_j+n-j}}
q^{|\mu+\delta_n|} s_\lambda(q^{\mu+\delta_n})\DeltaBar(q^{\mu+\delta_n}).
\]
\end{thm}

\begin{proof}
  Let 
\[
x=(x_1,x_2,\dots,x_n) = 
(q^{\mu_1+n-1},q^{\mu_2+n-2},\dots,q^{\mu_n})=q^{\mu+\delta_n}.
\]
Following the notations
  in Proposition~\ref{prop:schur_poset}, we have
\[
s_\lambda(x)\DeltaBar(x) = \prod_{j=1}^{n}[\lambda_j+n-j]_q!
\int_{\order_I(P_{\schur}(n,\lambda))} d_q x^{(\lambda_1+n-1)} d_q x^{(\lambda_1+n-2)} \cdots
d_q x^{(1)}.
\]

Let $\omega$ be the labeling of $P_{\schur}(n,\lambda)$ defined by
$\omega(x_i^{(k)})=j$ if $d_qx_i^{(k)}$ is in the $j$th position of
$d_q x^{(\lambda_1+n-1)} d_q x^{(\lambda_1+n-2)} \cdots d_q x^{(1)}$.
Then by Theorem~\ref{thm:order}, we have
\[
\int_{\order_I(P_{\schur}(n,\lambda))} d_q x^{(\lambda_1+n-1)} d_q x^{(\lambda_1+n-2)} \cdots
d_q x^{(1)} = (1-q)^{\binom n2 + |\lambda|}
\sum_{\sigma} q^{|\sigma|},
\]
where the sum is over all $(P_{\schur}(n,\lambda),\omega)$-partitions
$\sigma$ such that 
\[
  \begin{array}{ll}
\mu_j+n-j >\sigma(x_j^{(1)}) \ge \mu_{j+1}+n-j-1,\quad 1\le j\le n-1
    & \mbox{if $\ell(\lambda)\le n-1$,}\\
\mu_{j-1}+n-j+1> \sigma(x_j^{(1)}) \ge \mu_{j}+n-j,\quad 1\le j\le n
    & \mbox{if $\ell(\lambda)=n$,}
  \end{array}
\]
where $\mu_0 =\infty$.  There is an obvious way to regard such
$\sigma$ as a reverse plane partition
$U\in \RPP((\delta_{n+1}+\lambda)^*)$ with
$|\sigma|=|U|-|\mu+\delta_n|$ such that the
$\rdiag(U)=q^{\mu+\delta_n}$ and in each column the entries are
strictly increasing from top to bottom. Thus, by
Theorem~\ref{thm:order}, we have
\[
\int_{\order_I(P_{\schur}(n,\lambda))} d_q x^{(\lambda_1+n-1)} d_q x^{(\lambda_1+n-2)} \cdots
d_q x^{(1)}
=(1-q)^{|\lambda|+\binom{n}2} \sum_{U} q^{|U|-\sum_{i=1}^n(\mu_i+n-i)},
\]
where the sum is over all column-strict reverse plane partitions
$U\in \RPP((\delta_{n+1}+\lambda)^*)$ with
$\rdiag(U)=q^{\mu+\delta_n}$. For such $U$, let $T$ be the reverse
plane partition obtained from it by decreasing the entries in row $i$
by $i-1$ for $i=1,2,\dots,n$. Then
$T\in \RPP ( (\delta_{n+1}+\lambda)^* )$, $\rdiag(T)=\mu$, and
$|U|=|T|+b(\delta_{n+1}+\lambda)$.  Summarizing these, we obtain
\[
s_\lambda(x)\DeltaBar(x) = \prod_{j=1}^{n}[\lambda_j+n-j]_q!
(1-q)^{|\lambda|+\binom{n}2} \sum_{\substack{T\in \RPP ( (\delta_{n+1}+\lambda)^* ) \\
\rdiag(T)=\mu}} q^{|T|+b(\delta_{n+1}+\lambda) -\sum_{i=1}^n(\mu_i+n-i)},
\]
which is equivalent to the theorem.
\end{proof}

\subsection{Generalized Gelfand-Tsetlin patterns}

In this subsection we restate Theorem~\ref{thm:rpp} in terms of
generalized Gelfand-Tsetlin patterns, see Theorem~\ref{thm:GT}.

\begin{defn}
  For $\lambda,\mu\in\Par_n$, an \emph{$(n,\lambda,\mu)$-Gelfand-Tsetlin
    pattern} is a triangular array
\[
\{(x_{i,j}): 1\le j\le n,\quad 1-\lambda_j\le i \le n+1-j\}
\]
of nonnegative integers such that $x_{i,j+1}\ge x_{i,j} \ge x_{i+1,j}$
and $x_{i,n+1-i}=\mu_i$. We assume that the inequality holds if any
index lies outside the stated domain. We denote by $\GT_n(\lambda,\mu)$ the
set of $(n,\lambda,\mu)$-Gelfand-Tsetlin patterns.
\end{defn}

\begin{exam}
Let $n=4$, $\lambda=(3,1)$ and $\mu=(3,2,2,1)$. Then the coordinates
of an $(n,\lambda,\mu)$-Gelfand-Tsetlin pattern are arranged as
follows.
\[
\begin{matrix}
&&&x_{1,4}&&&\\
&&&x_{1,3} &x_{2,3}&&\\
&&x_{0,2} &x_{1,2} &x_{2,2}&x_{3,2}&\\
x_{-2,1} &x_{-1,1}&x_{0,1}&x_{1,1}&x_{2,1}&x_{3,1}&x_{4,1}
\end{matrix}
\]
An example of an $(n,\lambda,\mu)$-Gelfand-Tsetlin is shown below.
\[
\begin{matrix}
&&&3&&&\\
&&&3&2&&\\
&&4&3&2&2&\\
7&5&3&3&2&1&1
\end{matrix}
\]
\end{exam}

Note that if $\lambda=\varnothing$, then $\GT_n(\varnothing,\mu)$ is
basically the set of Gelfand-Tsetlin patterns, see \cite[(7.37),
p. 313]{EC2}.  There is a well known bijection \cite[p. 314]{EC2}
between Gelfand-Tsetlin patterns and column strict tableaux which
gives
\begin{equation}
  \label{eq:gt}
\sum_{T\in\GT_n(\varnothing,\mu)} q^{|T|} = 
q^{|\mu|} s_\mu(q^{\delta_n}).
\end{equation}

Note also that, if rotated by $180^\circ$, the
$(\lambda,\mu)$-Gelfand-Tsetlin patterns are the same as
the reverse plane partitions of shape $(\delta_{n+1}+\lambda)^*$ with
reverse diagonal entries given by $\mu$. Thus Theorem~\ref{thm:rpp} can be
restated as Theorem~\ref{thm:GT}, which generalizes \eqref{eq:gt} to
arbitrary $\lambda$.

\begin{thm}\label{thm:GT}
 For $\lambda,\mu\in\Par_n$, we have
\[
\sum_{T\in\GT_n(\lambda,\mu)} q^{|T|}
=  q^{|\mu|-b(\lambda)}
 s_\mu(q^{\delta_n}) s_\lambda(q^{\mu+\delta_n})
\prod_{j=1}^{n}\frac{(q;q)_{n-j}}{(q;q)_{\lambda_j+n-j}}.
\]
Equivalently, 
\[
\sum_{T\in\GT_n(\lambda,\mu)} q^{|T|}
=  \frac{q^{|\mu|-b(\delta_{n}+\lambda)}}
{\prod_{j=1}^{n}(q;q)_{\lambda_j+n-j}}
s_\lambda(q^{\mu+\delta_n})\DeltaBar(q^{\mu+\delta_n}).
\]
\end{thm}
\begin{proof}
By Theorem~\ref{thm:rpp}, we have
\[
\sum_{T\in\GT_n(\lambda,\mu)} q^{|T|}
=\sum_{\substack{T\in \RPP ( (\delta_{n+1}+\lambda)^* ) \\
\rdiag(T)=\mu}}  q^{|T|}
=  \frac{q^{-b(\delta_{n+1}+\lambda)}}
{\prod_{j=1}^{n}(q;q)_{\lambda_j+n-j}}
q^{|\mu+\delta_n|} s_\lambda(q^{\mu+\delta_n})\DeltaBar(q^{\mu+\delta_n}).
\]
Then we obtain the theorem using the fact \cite[(7.105)]{EC2}
\[
 s_\mu(q^{\delta_n}) = \prod_{1\le i<j\le n}
\frac{q^{\mu_j+n-j}-q^{\mu_i+n-i}}{q^{i-1}-q^{j-1}}
=\frac{\DeltaBar(q^{\mu+\delta_n})}
{q^{b(\delta_{n})} \prod_{j=1}^{n-1}(q;q)_{j}}.
\]
\end{proof}

\subsection{The trace-generating function for reverse plane partitions}

In this subsection we give two results on reverse plane partitions. 
In Corollary~\ref{cor:rpp2} we give a generating function for 
reverse plane partitions with a fixed shape and diagonal. 
Then we reinterpret an integral of Warnaar as the trace 
generating function for reverse plane partitions of a given shape,
Theorem~\ref{thm:warnaar_rpp} and Corollary~\ref{cor:rpp_nu}.

The first result gives the generating function for 
reverse plane partitions with a fixed shape and diagonal by decomposing 
via the Durfee square. Here, the
\emph{Durfee square} of a partition $\lambda$ is the largest square
containing the cell in row $1$ and column $1$ in $\lambda$.

\begin{cor}
\label{cor:rpp2}
  Let $\lambda=(\lambda_1,\dots,\lambda_n)$, $\mu=(\mu_1,\dots,\mu_n)$
  and $\nu$ the partition obtained from $(n^n)$ by attaching $\lambda$
  to the right and $\mu'$ at the bottom.  Then for a partition
  $\rho=(\rho_1,\dots,\rho_n)$, we have
  \begin{equation}\label{eq:34}
\sum_{\substack{T\in \RPP (\nu) \\
\rdiag(T)=\rho}}  q^{|T|}
=  \frac{q^{\binom n2 -b(\delta_{n+1}+\lambda) -b(\delta_{n+1}+\mu)}
q^{|\rho+\delta_n|}}
{\prod_{j=1}^{n}(q;q)_{\lambda_j+n-j}(q;q)_{\mu_j+n-j}}
s_\lambda(q^{\rho+\delta_n})s_\mu(q^{\rho+\delta_n})
\DeltaBar(q^{\rho+\delta_n})^2. 
  \end{equation}
\end{cor}
\begin{proof}
Since
\[
\sum_{\substack{T\in \RPP (\nu) \\
\rdiag(T)=\rho}}  q^{|T|}
=q^{-|\rho|}\sum_{\substack{T\in \RPP ( (\delta_{n+1}+\lambda)^* ) \\
\rdiag(T)=\rho}}  q^{|T|}
\sum_{\substack{T\in \RPP ( (\delta_{n+1}+\mu)^* ) \\
\rdiag(T)=\rho}}  q^{|T|},
\]
we get the desired formula by Theorem~\ref{thm:GT}.
\end{proof}

The following $q$-integral evaluation is the special case $k=1$ of Warnaar's 
integral
\cite[Theorem~1.4]{Warnaar2005}, which has Macdonald polynomials 
instead of Schur functions.

\begin{thm} \cite[Special case of Theorem~1.4]{Warnaar2005}
\label{thm:warnaar}
  Let $\lambda=(\lambda_1,\dots,\lambda_n)$ and
  $\mu=(\mu_1,\dots,\mu_n)$ be partitions,
  $\Re(\alpha)>-\lambda_n-\mu_n$ and $0<q<1$.  Then
  \begin{multline*}
\int_{0\le x_1\le \cdots\le x_n\le 1} s_\lambda(x)  s_\mu(x) 
\prod_{i=1}^n x_i^{\alpha-1} \DeltaBar(x)^2 d_qx
=(1-q)^n q^{\alpha\binom n2 + 2\binom n3} \\
\times s_\lambda(1,q,\dots,q^{n-1}) s_\mu(1,q,\dots,q^{n-1})
\prod_{i=1}^{n-1} (q;q)_i^2
\prod_{i,j=1}^n \frac{1}{1-q^{\alpha+2n-i-j+\lambda_i+\mu_j}}.
  \end{multline*}
\end{thm}
\begin{proof}
If $k=1$ in \cite[Theorem~1.4]{Warnaar2005}, we have 
  \begin{multline*}
\int_{[0,1]^n} s_\lambda(x)  s_\mu(x) 
\prod_{i=1}^n x_i^{\alpha-1}
\prod_{1\le i<j\le n}x_{j}^{2}\left(\frac{x_i}{x_j}\right)_{2}d_qx\\
=q^{\alpha \binom n2 + 2\binom n3} (1-q)^{n^2}
s_\lambda(1,q,\dots,q^{n-1}) s_\mu(1,q,\dots,q^{n-1}) \\
\times \prod_{i=1}^n \Gamma_q(i)\Gamma_q(i+1)
\prod_{i,j=1}^n \frac{1}{1-q^{\alpha+2n-i-j+\lambda_i+\mu_j}}.
  \end{multline*}
  By the same argument proving the equivalence of \eqref{eq:q-Selberg}
  and \eqref{eq:Askey}, one can show that these two integral
  evaluations are equivalent.
\end{proof}

We now show that Warnaar's $q$-integral Theorem~\ref{thm:warnaar} 
can be interpreted as a trace-generating function
of a reverse plane partition of given shape. First we show that the trace generating 
function is a $q$-integral. 

\begin{thm}
\label{thm:warnaar_rpp}  
  Let $\lambda=(\lambda_1,\dots,\lambda_n)$, $\mu=(\mu_1,\dots,\mu_n)$
  and $\nu$ the partition obtained from $(n^n)$ by attaching $\lambda$
  to the right and $\mu'$ at the bottom.  Then
\begin{multline*}
\sum_{T\in \RPP (\nu)}  q^{|T|} (q^a)^{\tr(T)}
= \frac{q^{(1-a)\binom n2 -b(\delta_{n+1}+\lambda) -b(\delta_{n+1}+\mu)}}
{\prod_{j=1}^{n}(q;q)_{\lambda_j+n-j}(q;q)_{\mu_j+n-j}}  \\
\times \frac{1}{(1-q)^n} \int_{0\le x_1\le \cdots\le x_n \le 1} 
s_\lambda(\vec x)s_\mu(\vec x) \prod_{i=1}^nx_i^{a} 
\DeltaBar(\vec x)^2 d_q\vec x.
\end{multline*}
\end{thm}

\begin{proof}
  Let $\rho=(\rho_1,\dots,\rho_n)$ be a fixed partition. For
  $T\in \RPP (\nu)$ with $\rdiag(T)=\rho$, we have
\begin{equation}
  \label{eq:atr}
q^{a\cdot \tr(T)} = q^{a|\rho|} = 
q^{-a\binom n2} \left(q^{|\rho+\delta_n|}\right)^a.
\end{equation}
Thus, by Corollary~\ref{cor:rpp2}, 
\begin{multline}
  \label{eq:35}
\sum_{\substack{T\in \RPP (\nu) \\
\rdiag(T)=\rho}}  q^{|T|} (q^a)^{\tr(T)} 
=  \frac{q^{(1-a)\binom n2 -b(\delta_{n+1}+\lambda) -b(\delta_{n+1}+\mu)}
\left(q^{|\rho+\delta_n|}\right)^{a+1}}
{\prod_{j=1}^{n}(q;q)_{\lambda_j+n-j}(q;q)_{\mu_j+n-j}}\\
\times s_\lambda(q^{\rho+\delta_n})
s_\mu(q^{\rho+\delta_n})
\DeltaBar(q^{\rho+\delta_n})^2. 
\end{multline}

By Lemma~\ref{lem:par}, if we take the sum of both sides of
\eqref{eq:35} over all $\rho\in\Par_n$, we get the theorem.
\end{proof}

As a corollary of Theorems~\ref{thm:warnaar} and \ref{thm:warnaar_rpp},
one can prove the 
trace-generating function for reverse plane partitions, which is a
special case of \cite[Theorem~5.1]{Gansner1981}. One uses the principally 
specialized Schur functions as products, 
\cite[Lemma~7.21.1]{EC2} and \cite[Lemma~7.21.2]{EC2}.
We do not give the details.

\begin{cor}\label{cor:rpp_nu}
Let $\nu$ be a partition. Then 
\[
\sum_{T\in \RPP (\nu)}   (q^a)^{\tr(T)} q^{|T|} = 
\prod_{u\in\nu} \frac{1}{1-(q^a)^{\chi(u)} q^{h(u)}},
\]
where $\chi(u)=1$ if $u$ is in the Durfee square of $\nu$ and
$\chi(u)=0$ otherwise. 
\end{cor}

\subsection{The trace-generating function for shifted reverse plane
  partitions}

In this subsection we find a ``shifted''-counterpart to Warnaar's
$q$-integral, which gives a trace generating function for shifted
reverse plane partitions.

We need a special case of Gansner's result on the trace-generating
function for shifted reverse plane partitions. We note that 
Gansner \cite{Gansner1981} considered a more general weight which
involves the entries in the $k$-diagonal for each $k$. 

\begin{thm}\cite[A special case of Theorem~7.1]{Gansner1981}
\label{thm:gansner}
For $\lambda\in\Par_n$, we have
\[
\sum_{T\in\RPP((\lambda+\delta_{n+1})^*)}x^{\tr(T)} q^{|T|}
=\prod_{u\in (\lambda+\delta_{n+1})^*}
\frac{1}{1-x^{\chi(u)}q^{h(u)}},
\]
where $\chi(u)=1$ if $u$ is in column $j$ for $1\le j\le n$, and
$\chi(u)=0$ otherwise, and $h(u)$ is the length of the shifted hook at
$u$, see \cite{Gansner1981} for the definition.  Equivalently, this
can be restated as
\begin{multline*}
\sum_{T\in\RPP((\lambda+\delta_{n+1})^*)}x^{\tr(T)} q^{|T|}
=q^{-b(\lambda)} s_\lambda(1,q,\dots,q^{n-1}) \prod_{j=1}^n
\frac{(q;q)_{j-1}}{(q;q)_{\lambda_j+n-j}}\\
\times \prod_{1\le i\le j\le n} 
\frac{1}{1-xq^{1+2n-i-j+\lambda_i+\lambda_{j+1}}}.
\end{multline*}
\end{thm}

We do not prove the equivalence of the second statement. It follows from 
routine facts on Schur functions. 

Theorem~\ref{thm:gansner} gives the generating function for
shifted reverse plane partitions, but we want
to find the corresponding $q$-integral.  Here is the ``shifted''
version of Warnaar's $q$-integral which corresponds to Gansner's
theorem.

\begin{thm}
\label{thm:gansner_int}
If $\lambda\in \Par_n$,
  \begin{multline*}
\int_{0\le x_1\le \cdots\le x_n\le 1} s_\lambda(x)  
\prod_{i=1}^n x_i^{\alpha-1} \DeltaBar(x) d_qx
=(1-q)^n q^{\alpha\binom n2 + \binom n3} s_\lambda(1,q,\dots,q^{n-1})\\
\times \prod_{i=1}^{n-1} (q;q)_i \prod_{1\le i\le j\le n} 
\frac{1}{1-q^{\alpha+2n-i-j+\lambda_i+\lambda_{j+1}}}.
  \end{multline*}
\end{thm}
\begin{proof}
  Let $a=\alpha-1$.  By Theorem~\ref{thm:GT}, we have
\[
\sum_{\substack{T\in \RPP ( (\delta_{n+1}+\lambda)^* ) \\
\rdiag(T)=\mu}}  (q^a)^{\tr(T)}q^{|T|}
= q^{-a\binom n2} q^{a|\mu+\delta_n|} \frac{q^{-b(\delta_{n+1}+\lambda)}}
{\prod_{j=1}^{n}(q;q)_{\lambda_j+n-j}}
q^{|\mu+\delta_n|} s_\lambda(q^{\mu+\delta_n})\DeltaBar(q^{\mu+\delta_n}).
\]
By summing both sides over all $\mu\in\Par_n$ and using 
Lemma~\ref{lem:par}, we obtain
\begin{multline*}
\sum_{T\in \RPP ( (\delta_{n+1}+\lambda)^* )}  (q^a)^{\tr(T)}q^{|T|}
= \frac{q^{-a\binom n2 -b(\delta_{n+1}+\lambda)}}
{\prod_{j=1}^{n}(q;q)_{\lambda_j+n-j}} 
\sum_{\mu\in\Par_n} q^{|\mu+\delta_n|} q^{a|\mu+\delta_n|}  
s_\lambda(q^{\mu+\delta_n})\DeltaBar(q^{\mu+\delta_n})\\
= \frac{q^{-(a+1)\binom n2 -\binom n3- b(\lambda)}}
{\prod_{j=1}^{n}(q;q)_{\lambda_j+n-j}} 
\frac{1}{(1-q)^n} \int_{0\le x_1\le \cdots\le x_n\le 1} s_\lambda(x)  
\prod_{i=1}^n x_i^{a} \DeltaBar(x) d_qx.
\end{multline*}
We finish the proof by comparing this with Theorem~\ref{thm:gansner}.  
\end{proof}

\subsection{Reverse plane partitions and the $q$-Selberg integral}

In this subsection we show that the Askey-Kadell-Selberg integral is
equivalent to a generating function for reverse plane partitions of a
square shape with certain weight. The weight has 3 parameters and
generalizes the sum of the entries of a reverse plane partition.

\begin{figure}
  \centering
\includegraphics{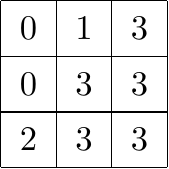} 
  \caption{A reverse plane partition of square shape $(3^3)$.}
\label{fig:rpp}
\end{figure}

\begin{defn}
  For a reverse plane partition $T\in\RPP(n^n)$ with
  $\rdiag(T)=(\mu_1,\dots,\mu_n)$, let $v_i=\mu_i+n-i$ for $i=1,2,\dots,n$.
  For integers $a,b\ge0$ and $m\ge1$, we define
\[
\wt_{a,b,m}(T) = q^{|T|+a\cdot\tr(T)}
(q^{v_1+1} ,\dots, q^{v_n+1};q)_b
\prod_{1\le i<j\le n} 
\frac{(q^{v_j})^{2m-1} (q^{1-m+v_i-v_j};q)_{2m-1}}{q^{v_j}-q^{v_i}}.
\]
\end{defn}

\begin{exam}
  Let $T$ be the reverse plane partition in Figure~\ref{fig:rpp}. Then
  $|T|=18$, $\tr(T)=6$, $\rdiag(T)=(3,3,0)$ and
  $(v_1,v_2,v_3)=(3,3,0)+(2,1,0)=(5,4,0)$. Thus
  \begin{multline*}
\wt_{a,b,m}(T) = q^{18+6a} (q^6,q^5,q^1;q)_b
\cdot \frac{(q^4)^{2m-1}(q^{1-m+5-4};q)_{2m-1}}{q^4-q^5}    
\cdot \frac{(q^0)^{2m-1}(q^{1-m+5-0};q)_{2m-1}}{q^0-q^5}    \\
\times
 \frac{(q^0)^{2m-1}(q^{1-m+4-0};q)_{2m-1}}{q^0-q^4}    \\
=q^{18+6a} (q^6,q^5,q^1;q)_b
\cdot \frac{q^{8m-8}(q^{2-m},q^{6-m},q^{5-m};q)_{2m-1}}
{(1-q)(1-q^5)(1-q^4)} .
  \end{multline*}
\end{exam}

Note that
\begin{align*}
\wt_{a,b,1}(T) &= q^{|T|+a\cdot\tr(T)} (q^{v_1+1} ,\dots, q^{v_n+1};q)_b,\\
\wt_{a,0,1}(T) &= q^{|T|+a\cdot\tr(T)}, \\
\wt_{0,0,1}(T) &= q^{|T|}.
\end{align*}

The following theorem shows that the Askey-Kadell-Selberg integral
is a generating function for reverse plane partitions of a square
shape, up to simple constants.

\begin{thm}
\label{thm:AKS_RPP}
The weighted generating function for reverse plane partitions of a square shape
is given by the Askey-Kadell-Selberg integral 
\[
\sum_{T\in \RPP(n^n)} \wt_{a,b,m}(T) =
\frac{q^{(-1-a)\binom n2 - 2\binom n3}}{(1-q)^n\prod_{j=1}^{n-1}(q;q)_{j}^2}
\cdot S_q(n,a+1,b+1,m).
\]
\end{thm}
\begin{proof}
  Let us fix $\mu=(\mu_1,\dots,\mu_n)\in\Par_n$ and
  $v=(v_1,\dots,v_n)$ with $v_i=\mu_i+n-i$. Then
\[
\sum_{\substack{T\in \RPP(n^n)\\\rdiag(T)=\mu}}  \wt_{a,b,m}(T) = 
\sum_{\substack{T\in \RPP(n^n)\\\rdiag(T)=\mu}} 
q^{|T|+a\cdot \tr(T)} C(q^{\mu+\delta_n}),
\]
where
\[
C(\vec x) = \prod_{i=1}^n (qx_i;q)_b
\prod_{1\le i<j\le n}  x_j^{2m-1}(q^{1-m}x_i/x_j;q)_{2m-1}
\DeltaBar(\vec x)^{-1}.
\]
By the special case $\lambda = \mu = (0^n)$ of
Corollary~\ref{cor:rpp2} and \eqref{eq:atr}, we have
\[
\sum_{\substack{T\in \RPP (n^n) \\
\rdiag(T)=\mu}}  q^{|T|+a\cdot \tr(T)}
=  \frac{q^{(1-a)\binom n2 - 2b(\delta_{n+1})} (q^{|\mu+\delta_n|})^{a+1}}
{\prod_{j=1}^{n-1}(q;q)_{j}^2} \DeltaBar(q^{\mu+\delta_n})^2. 
\]
Combining these two equations we get
\[
\sum_{\substack{T\in \RPP(n^n)\\\rdiag(T)=\mu}}  \wt_{a,b,m}(T) = 
\frac{q^{(1-a)\binom n2 - 2b(\delta_{n+1})} (q^{|\mu+\delta_n|})^{a+1}}
{\prod_{j=1}^{n-1}(q;q)_{j}^2} \DeltaBar(q^{\mu+\delta_n})^2
C(q^{\mu+\delta_n}).
\]
By summing over all $\mu\in\Par_n$ and using Lemma~\ref{lem:par},
we get
\[
\sum_{T\in \RPP(n^n)} \wt_{a,b,m}(T) 
=\frac{q^{(1-a)\binom n2 - 2b(\delta_{n+1})}}{\prod_{j=1}^{n-1}(q;q)_{j}^2}
\frac{1}{(1-q)^n} S_q(n,a+1,b+1,m).
\]
Then we finish the proof using the fact
$b(\delta_{n+1})=\binom n2 +\binom n3$.
\end{proof}

Using Theorems~\ref{thm:AKS_RPP} and \ref{thm:AHKS}, 
we obtain a product formula for the weighted generating function.

\begin{thm}
\label{thm:selbergRPP}
The weighted generating function for reverse plane partitions of square shape
has the explicit product formula
  \begin{multline*}
\sum_{T\in \RPP(n^n)} \wt_{a,b,m}(T) =
\frac{q^{(1-a+m+am)\binom n2 -2\binom{n+1}3 +2m^2 \binom
    n3}}{(1-q)^{n^2}} \\ \times
\prod_{j=1}^n \frac{[a+(j-1)m]_q! [b+(j-1)m]_q! [jm-1]_q!}
{[a+b+(n+j-2)m+1]_q![m-1]_q! ([j-1]_q!)^2}.
  \end{multline*}
\end{thm}

We conclude with some special cases of Theorem~\ref{thm:selbergRPP}.

If $m=1$ in Theorem~\ref{thm:selbergRPP}, we have
\begin{cor}
\label{cor:selbergRPP}
A weighted generating function for reverse plane partitions of square shape is
\[
\sum_{T\in \RPP(n^n)} q^{|T|+a\cdot\tr(T)} 
(q^{v_1+1},\dots,q^{v_n+1};q)_b=
\frac{1}{(1-q)^{n^2}}
\prod_{j=1}^n \frac{[a+j-1]_q! [b+j-1]_q!}
{[a+b+n+j-1]_q! [j-1]_q!}.
\]
where the reverse diagonal entries of $T$ are 
$(\mu_1,\dots, \mu_n)$, and $v_i=\mu_i+n-i$, $1\le i\le n$.  
\end{cor}

If $b=0$ in Corollary~\ref{cor:selbergRPP}, we have 
the trace generating function for square shapes, 
see \cite{EC2}. If also $a=0$, we have the generating function for reverse 
plane partitions of square shape.

Kamioka \cite{Kamioka2015} has another weighted generating function for reverse plane partitions
of a square shape.

\section{$q$-Ehrhart polynomials}
\label{sec:q-ehrh-polyn}

In this section, using $q$-integrals, we study $q$-Ehrhart polynomials
and $q$-Ehrhart series of order polytopes with some faces removed.  We
refer the reader to \cite{Chapoton2016} for the more details in
$q$-Ehrhart polynomials of lattice polytopes.  This section was
initiated by discussions with \JV..

Throughout this section we assume that $P$ is a poset on
$\{x_1,\dots,x_n\}$ and $\omega_n$ is the bijection from
$\{x_1,\dots,x_n\}$ to $[n]$ given by $\omega_n(x_i)=i$ for
$1\le i\le n$.

We denote by $\overline P$ the dual of $P$, i.e., the poset obtained
by reversing the orders in $P$.

\begin{defn}
  For a bounded set $X$ of points in $\mathbb{R}^n$ and a positive
  integer $m$, we define the \emph{$q$-Ehrhart function}
\[
E_q(X,m) = \sum_{(x_1,\dots,x_n)\in mX \cap \mathbb{Z}^n} 
q^{x_1+\cdots+x_n}.
\]
\end{defn}

\begin{defn}
\label{defn:deltaP}
Let $P$ be a poset on $\{x_1,\dots,x_n\}$. We
define $\Delta(P)$ to be the set of points
$(x_1,\dots,x_n)\in [0,1]^n$ such that
\begin{itemize}
\item $x_i \ge x_j$ if $x_i\le _P x_j$,
\item $x_i > x_j$ if $x_i\le _P x_j$ and $i>j$.
\end{itemize}
\end{defn}

Note that $\Delta(P)$ is obtained from the order polytope
$\order(\overline P)$ by removing some faces. In general $\Delta(P)$
is not a polytope. However, if $(P,\omega_n)$ is naturally labeled,
then $\Delta(P)=\order(\overline P)$. Therefore, by adding the
assumption that $(P,\omega_n)$ is naturally labeled, every result in
this section has a corollary stated in terms of the order polytope
$\order(\overline P)$. For instance, see
Corollary~\ref{cor:ehrhart_poset}.

Now we show some properties of the $q$-Ehrhart function of
$\Delta(P)$: it is represented as a $q$-integral of a truncated order
polytope of $P$ and it is a polynomial in $[m]_q$ whose leading
coefficient is the $q$-volume of the order polytope of $P$.

\begin{thm}\label{thm:ehrhart_poset}
  Let $P$ be a poset on $\{x_1,\dots,x_n\}$.  Then, for an integer
  $m\ge0$, we have
\[
E_q(\Delta(P), m)  
= \frac{1}{(1-q)^n}V_q(\order_{[q^{m+1},1]^n}(P)),
\]
and, equivalently, 
\[
E_q(\Delta(P), m)  = \sum_{\pi\in\LL(P,\omega_n)} 
q^{\maj(\pi)} \qbinom{m+n-\des(\pi)}{n}.
\]
Moreover, the $q$-Ehrhart function $E_q(\Delta(P), m)$ is a
polynomial in $[m]_q$ whose coefficients are rational functions in $q$
and whose leading coefficient is the $q$-volume $V_q(\order(P))$.
\end{thm}
\begin{proof}
By definition, we have
\[
E_q(\Delta(P),m) = \sum_{\sigma} q^{|\sigma|},
\]
where the sum is over all $(P,\omega_n)$-partitions $\sigma$ with
$\max(\sigma)\le m$. By Theorem~\ref{thm:order}, this is equal to
$\frac{1}{(1-q)^n}V_q(\order_{[q^{m+1},1]^n}(P))$. 

By Corollary~\ref{cor:LL(P)}, we have
\[
V_q(\order_{[q^{m+1},1]^n}(P))
=\sum_{\pi\in\LL(P,\omega_n)} V_q(\order_{[q^{m+1},1]^n}(P_\pi)).
\]
By Corollary~\ref{cor:ab},  we have
\[
V_q(\order_{[q^{m+1},1]^n}(P_\pi))
=(1-q)^n q^{\maj(\pi)} \qbinom{m+n-\des(\pi)}{n}, 
\]
which is a polynomial in $[m]_q$ whose coefficients are rational
functions of $q$. Since $E_q(\Delta(P),m)$ is a sum of these
$q$-volumes divided by $(1-q)^n$, it is also a polynomial in $[m]_q$.
The leading coefficient of $E_q(\Delta(P), m)$ as a polynomial in
$[m]_q$ is
\[
\lim_{m\to\infty} \frac{E_q(\Delta(P), m)}{[m]_q^n}
=\lim_{m\to\infty} \frac{1}{(1-q^m)^n}
V_q(\order_{[q^{m+1},1]^n}(P))
=V_q(\order(P)).
\]
\end{proof}

If we add the assumption that $(P,\omega_n)$ is naturally labeled in
Theorem~\ref{thm:ehrhart_poset} then we obtain a result on the order
polytope $\order(\overline P)$.

\begin{cor}\label{cor:ehrhart_poset}
  Let $P$ be a poset on $\{x_1,\dots,x_n\}$. Suppose that
  $(P,\omega_n)$ is naturally labeled.  Then, for an integer $m\ge0$,
  we have
\[
E_q(\order(\overline P), m) =
\frac{1}{(1-q)^n}V_q(\order_{[q^{m+1},1]^n}(P)).
\]
Moreover, the $q$-Ehrhart function $E_q(\order(\overline P), m)$ is a
polynomial in $[m]_q$ whose coefficients are rational functions in $q$
and whose leading coefficient is the $q$-volume $V_q(\order(P))$.
\end{cor}

We note that $E_q(X,m)$ is not always a polynomial in $[m]_q$. For
example, if $n=1$ and $X=\{1/k\}$ then $E_q(X,m)=q^{m/k}$ if $m$ is
divisible by $k$ and $E_q(X,m)=0$ otherwise.

\begin{remark}
  In \cite{Chapoton2016}, Chapoton defines the $q$-volume of a poset
  $P$ on $\{x_1,\dots,x_n\}$ to be the leading coefficient of the
  $q$-Ehrhart polynomial $E_q(\order(P),m)$ times $[n]_q!$.  If
  $(\overline P,\omega_n)$ is naturally labeled, we have
  $\Delta(\overline P)=\order(P)$.  By
  Corollary~\ref{cor:ehrhart_poset}, the leading coefficient of the
  $q$-Ehrhart polynomial $E_q(\order(P),m)$ is the $q$-volume
  $V_q(\order(\overline P))$. Thus, Chapoton's $q$-volume of $P$ is
  equal to $[n]_q!V_q(\order(\overline P))$.
\end{remark}

Next we consider the $q$-Ehrhart series.

\begin{defn}
  For a set $X$ of points in $\mathbb R^n$, we define the \emph{$q$-Ehrhart
  series} of $X$ by
\[
E_q^*(X,t) = \sum_{m\ge0} E_q(X,m) t^m.
\]
\end{defn}

We now show that the $q$-Ehrhart series of $\Delta(P)$ is a generating
function for the linear extensions of $P$. 

\begin{cor}\label{cor:ehr}
  For a poset $P$ on $\{x_1,\dots,x_n\}$, we have
\[
E_q^*(\Delta(P), t) = \frac{1}{(t;q)_{n+1}}
\sum_{\pi\in \LL({P},\omega_n)}t^{\des(\pi)}q^{\maj({\pi})}.
\]
\end{cor}
\begin{proof}
By Theorem~\ref{thm:ehrhart_poset}, we have
\[
E_q^*(\Delta(P), t) = \sum_{m\ge0} t^m
\sum_{\pi\in\LL(P,\omega_n)} 
q^{\maj(\pi)} \qbinom{m+n-\des(\pi)}{n}.
\]  
Thus, it suffices to show that
\begin{equation}
  \label{eq:1}
\sum_{m\ge0} t^m \qbinom{m+n-\des(\pi)}{n}
=\frac{t^{\des(\pi)}}{(t;q)_{n+1}}. 
\end{equation}
The summand in the left hand side of \eqref{eq:1} is zero unless
$m\ge\des(\pi)$. By shifting the index $m$ by $m+\des(\pi)$, 
the left hand side of \eqref{eq:1} becomes
\[
t^{\des(\pi)} \sum_{m\ge0} t^{m} \qbinom{m+n}{n}
=t^{\des(\pi)} \cdot \frac{1}{(t;q)_{n+1}}, 
\]
where the $q$-binomial theorem is used. This completes the proof.
\end{proof}

Let us consider the special case of Corollary~\ref{cor:ehr} when $P$ is
the anti-chain on $\{x_1,\dots,x_n\}$. In this case we have
$\Delta(P)=[0,1]^n$ and $\LL(P,\omega_n)=S_n$. Since
$E_q([0,1]^n,m) = (1+q+\cdots+q^m)^n$, Corollary~\ref{cor:ehr} reduces
to MacMahon's identity
\[
\sum_{\pi\in S_n}t^{\des(\pi)}q^{\maj({\pi})}
=\sum_{m\ge0}  t^m (1+q+\cdots+q^m)^n.
\]
We note that Beck and Braun \cite{Beck2013} also gave a proof this
identity by decomposing $[0,1]^n$ into simplices.

We finish this section by showing that the $q$-Ehrhart series of
$\Delta(P)$ has a $q$-integral representation. Note that if $0<t<1$
and $0<q<1$,
then $\log_q t>0$. 

\begin{cor}\label{cor:ehr2}
  Let $0<t<1$ and $0<q<1$.  For a poset $P$ on $\{x_1,\dots,x_n\}$, we have
\[
E_q^*(\Delta(P), t) = \frac{1}{(1-q)^{n+1}}
\int_{\order(P')} 
x_0^{\log_q t-1} d_q x_0 d_qx_{1}\cdots d_qx_{n}.
\]
where $P'$ is the poset obtained from ${P}$ by adding a new element
$x_0$ which is smaller than all elements in ${P}$.   
\end{cor}
\begin{proof}
  By Theorems~\ref{thm:ehrhart_poset} and \ref{thm:order}, 
\[
E_q^*(\Delta(P), t) = \sum_{m\ge0} t^m \cdot
\frac{1}{(1-q)^n}V_q(\order_{[q^{m+1},1]^n}(P)) 
=\sum_{m\ge0} t^m \sum_{\sigma} q^{|\sigma|},
\]
where the second sum in the rightmost side is over all
$(P,\omega_n)$-partitions $\sigma$ with $\max(\sigma)\le m$. 
This can be rewritten as
\[
\sum_{\sigma'} t^{\sigma'(x_0)} q^{|\sigma'|-\sigma'(x_0)}
=\sum_{\sigma'} (q^{\sigma'(x_0)})^{\log_q t-1} q^{|\sigma'|},
\]
where the sum is over all $(P',\omega')$-partitions $\sigma'$ for the
labeling $\omega'$ of $P'$ given by $\omega'(x_i)=i$ for
$0\le i\le n$.  By applying Theorem~\ref{thm:order} again, we obtain
the desired identity.
\end{proof}

\section*{Acknowledgements}

The authors are grateful to Matthieu \JV. for helpful discussion. 
They also thank Ira Gessel for pointing out MacMahon's result in
Theorem~\ref{thm:carlitz} and Ole Warnaar for helpful comments.


\end{document}